\documentclass{article}

\usepackage{arxiv}

\usepackage{natbib}

\usepackage{microtype}
\usepackage{graphicx}
\usepackage{subfigure}
\usepackage{booktabs} 
\usepackage{hyperref}

\usepackage{amsmath}
\usepackage{amssymb}
\usepackage{mathtools}
\usepackage{amsthm}

\usepackage[capitalize,noabbrev]{cleveref}

\theoremstyle{plain}
\newtheorem{theorem}{Theorem}[section]
\newtheorem{proposition}[theorem]{Proposition}

\theoremstyle{definition}
\newtheorem{definition}[theorem]{Definition}

\theoremstyle{remark}

\theoremstyle{remark}
\newtheorem{example}[theorem]{Example}

\usepackage[textsize=tiny]{todonotes}

\usepackage{algorithm}
\usepackage{algorithmic}
\usepackage{nicefrac} 
\usepackage{enumitem}
\usepackage{color}
\usepackage{multirow}
\usepackage{nccmath, amssymb}
\usepackage{bbold}
\usepackage{thmtools,thm-restate}
\usepackage{dblfloatfix} 
\usepackage{graphicx} % Required for inserting images
\usepackage{colortbl}
\usepackage{booktabs}
\usepackage{amsfonts}       % blackboard math symbols
\usepackage{amsthm}
\usepackage{amsmath}
\usepackage{amssymb}
\usepackage{xcolor}         % colors
\usepackage{xspace}
\usepackage{xparse}         % for \NewDocumentCommand
\usepackage{cleveref}       % \cref \Cref commands
\usepackage{stmaryrd}       % For \llbracket etc
\usepackage{bm}
\usepackage{nccmath, amssymb}
\usepackage{cases}
\usepackage{bbold}
\usepackage{pifont}% http://ctan.org/pkg/pifont
\DeclareMathOperator*{\dom}{\ensuremath{\text{\rm dom}}}

\DeclareMathOperator*{\range}{\ensuremath{\text{\rm Im}}}

\DeclareMathOperator*{\cl}{\ensuremath{\text{\rm cl}}}

\DeclareMathOperator*{\Span}{\ensuremath{\text{\rm span}}}

\DeclareMathOperator*{\Spec}{\ensuremath{\text{\rm Sp}}}

\DeclareMathOperator*{\fov}{\ensuremath{\text{ \rm W}}}

\newcommand{\Id}{I}

\providecommand{\norm}[1]{\ensuremath{\left\|#1\right\|}}
\providecommand{\SVDr}[1]{[\![#1]\!]_r}
\providecommand{\abs}[1]{\lvert#1\rvert}
\newcommand{\scalarp}[1]{{\langle #1\rangle}}
\newcommand{\R}{\mathbb R}
\newcommand{\C}{\mathbb C}
\newcommand{\N}{\mathbb N}
\newcommand{\bigO}{\mathcal O}

\newcommand{\EE}{\ensuremath{\mathbb E}}
\newcommand{\PP}{\ensuremath{\mathbb P}}

\newcommand{\HS}[1]{{\rm{HS}}\left(#1\right)} %Hilbert-Schmidt space on the domain

\newcommand{\im}{\pi} %invariant measure
\newcommand{\TOp}{A}  %Transfer operator on L2
\newcommand{\generator}{L} %Infinitesimal generator
\newcommand{\dt}{\Delta t} %sampling time
\newcommand{\spX}{\mathcal{X}} % space X
\newcommand{\RKHS}{\mathcal{H}} % RKHS
\newcommand{\Lii}{\mathcal{L}^2_\im(\spX)} % L2 space 
\newcommand{\Liishort}{\mathcal{L}^2_\im} % L2 space 
\newcommand{\Liio}{\mathcal{L}^{2}_{\im,0}(\spX)} % L2 space 
\newcommand{\Liioshort}{\mathcal{L}^{2}_{\im,0}} % L2 space 
\newcommand{\Wii}{\mathcal{W}^{1,2}_\im(\spX)} % Sobolev space

\newcommand{\Wiio}{\mathcal{W}^{1,2}_{\im,0}(\spX)} % Sobolev space

\newcommand{\scalarpH}[1]{{\langle #1\rangle}_{\RKHS}}

\newcommand{\shift}{\mu}
\newcommand{\weight}{a}
\newcommand{\reg}{\gamma}

\newcommand{\fH}{\phi}
\newcommand{\IfH}{{\psi}_{\weight}}

\newcommand{\HKoop}{\Estim_{\RKHS}}  %Koopman operator RKHS
\newcommand{\LKoop}{A}

\newcommand{\Cx}{C}
\newcommand{\Yx}{W_\weight}

\newcommand{\Cxy}[1]{\Cx_{#1}}

\newcommand{\Rx}{R_\shift}

\newcommand{\Data}{\mathcal{D}_n}

\newcommand{\ECx}{\widehat{C} } %Empirical Covariance of inputs 
\newcommand{\EYx}{\widehat{W}_\weight} %Emp Covariance of outputs 
\newcommand{\ECxy}[1]{\widehat{C}_{#1}}

\newcommand{\toeplitz}{\textsc{T}}
\newcommand{\toepsymb}{T}
\newcommand{\anfun}{F}

\newcommand{\Kx}{\textsc{K}}
\newcommand{\Kreg}{\Kx_{\reg}}

\newcommand{\ET}{\toeplitz_n}
\newcommand{\Estim}{G}  %Estimator
\newcommand{\EEstim}{\widehat{G}}  %Empirical estimator of the Koopman operator

\newcommand{\RRR}{\Estim^{r}_{\shift,\reg}}  %Estimator RRR
\newcommand{\ERRR}{\EEstim^{r}_{\weight,\reg}}  %Emp. Estimator RRR

\newcommand{\TS}{S_\im}  % Cannonical injection into L2
\newcommand{\ES}{\widehat{S}} % Sampling operator

\newcommand{\TJ}{J_\im}
\newcommand{\EJ}{\textsc{J}_n}

\newcommand{\Risk}{\mathcal{R}} % Risk
\newcommand{\ExRisk}{\mathcal{R}_{\rm ex}} % Exces Risk
\newcommand{\levec}{w^{\ell}}
\newcommand{\revec}{w^{r}}

\newcommand{\erefun}{\widehat{h}}
\newcommand{\lefun}{g}
\newcommand{\elefun}{\widehat{g}}
\newcommand{\gefun}{f}

\newcommand{\lgefun}{g}

\newcommand{\eval}{\lambda}
\newcommand{\geval}{\lambda}
\newcommand{\feval}{\nu}
\newcommand{\efeval}{\widehat{\feval}}

\newcommand{\eeval}{\widehat{\eval}}

\newcommand{\one}{\mathbb 1}

\newcommand{\U}{\textsc{U}}
\newcommand{\V}{\textsc{V}}
\newcommand{\W}{\textsc{W}}
\newcommand{\Z}{\textsc{Z}}

\newcommand{\betaLang}{(k_bT/\gamma)}

\usepackage[textsize=tiny]{todonotes}

\begin{document}

\title{Toeplitz Based Spectral Methods \\ for Data-driven Dynamical Systems}




\author{Vladimir R. Kostic \\ {Istituto Italiano di Tecnologia} \\ {University of Novi Sad} \\ {\tt \small vladimir.kostic@iit.it}\\
\And
Karim Lounici \\
CMAP, École Polytechnique \\
{\tt \small karim.lounici@polytechnique.edu}\\
\And
Massimiliano Pontil \\ {Istituto Italiano di Tecnologia} \\ {University College London} \\ {\tt \small massimiliano.pontil@iit.it}
}

%\icmlcorrespondingauthor{Firstname2 Lastname2}{first2.last2@www.uk}





\maketitle

\begin{abstract}
We introduce a Toeplitz-based framework for data-driven spectral estimation of linear evolution operators in dynamical systems. Focusing on transfer and Koopman operators from equilibrium trajectories without access to the underlying equations of motion, our method applies Toeplitz filters to the infinitesimal generator to extract eigenvalues, eigenfunctions, and spectral measures. Structural prior knowledge, such as self-adjointness or skew-symmetry, can be incorporated by design. The approach is statistically consistent and computationally efficient, leveraging both primal and dual algorithms commonly used in statistical learning.  Numerical experiments on deterministic and chaotic systems demonstrate that the framework can recover spectral properties beyond the reach of standard data-driven methods. 
\end{abstract}

\section{Introduction}

Computing eigenvalues and eigenfunctions of linear, possibly differential, operators is centerpiece in numerical analysis. Classical approaches are typically based on discretization, notably finite element methods. However, these methods suffer from the curse of dimensionality and in high dimensions data-driven approaches or Montecarlo simulations are often preferred.

In this paper we focus on a class of linear operators associated with Markov processes, such as transfer or Koopman operators, which evolve functions of the state (observables) over time.  These operators are associated with stochastic, ordinary or partial differential equations, including the Langevin dynamics and the Navier Stoke equations, and their spectral structure is key to understanding global system properties and forecasting future states. Important examples arise in computational physics, such as molecular dynamics \citep{SCHUTTE2003699} and climate modeling \citep{majda2009normal}, as well as in finance \citep{karatzas1991brownian}, among many others.
%In particular, the fields of molecular dynamics has particularly benefited from spectral decomposition methods of Markov semigroups. \cite{Schutte2001,devergne2024biased} to demonstrate the effectiveness of IG-based methods in accelerating simulations and enabling the practical identification of metastable states. Likewise in climate modelling data-driven approches could be unlock... 

While there is a large body of work on solving such eigenvalue problems using classical numerical methods, here we address scenarios in which {\em the equations of motion are unknown}, precluding direct discretization-based approaches. Moreover, we consider potentially high dimensional systems, where data driven approaches offer a means to overcome the curse of dimensionality \citep{Kostic2024diffusion}. We consider the setting in which one or multiple trajectories of the system at equilibrium are available, from which we wish to learn the operator spectrum. Current data-driven methods such as DMD, eDMD, or RRR, are designed to estimate the transfer or Koopman operators from data \citep[see, e.g.,][and references therein]{Bevanda2021,Brunton2022,Das2020,klus2018data,kohne2025error,Kostic_2023_learning,philipp2025error}. Many of these approaches can be interpreted as projection or Arnoldi-type schemes, closely related to Krylov subspace methods %\massi{more refs?} 
in numerical linear algebra \citep[see, e.g.,][]{Rowley2009}. However, when the time lag is very small, a computational bottleneck arises in the learning process, which has led to alternative approaches based on resolvent operators and Laplace transforms \citep{kostic2025laplace}. Related frequency-domain and resolvent-based perspectives for Koopman spectral analysis have also been explored in \citep{giannakis2019data}, among others.

These developments motivate us to consider a spectral estimation framework that encompasses more general analytic transforms like the exponential or the shifted inverse, to extract eigenvalues and eigenfunctions, when they exist, or spectral measures for deterministic systems in chaotic regimes. While this problem remains open in full generality, here we bridge well-established numerical methods 
%that build Krylov subspaces by 
based on 
%polynomial or 
Toeplitz linear algebra in order to design data-driven spectral estimation methods for dynamical systems observed in their stationary regime.

Starting from a Toeplitz symbol $\toepsymb$ on the unit circle we build filters that we apply to the generator $\generator$ of a dynamical system at equilibrium. Specifically the main idea is, given a time-lag $\dt>0$, to estimate the operator $\anfun(L):=\toepsymb(A_{\dt})$, 
%operator $A_{\dt} = e^{\dt \generator}$ being 
where $A_{\dt} = e^{\dt \generator}$ is 
the transfer operator at time-lag $\dt$, from a trajectory of equally $\dt$-time spaced data $\Data=\{x_i\}_{i=1}^{n}$ obtained from a stationary distribution of the system.  
Since the data sequence is a realization of random variables $(X_{i\dt})_{i\in\N}$, we fix a representation $\phi$ mapping each $X_{i\dt}$ to a function $\phi(X_{i\dt})$ in some hypothesis class of functions $\RKHS$, and observe that the adjoint of the transfer operator $A_{\dt}$ acts as the \textit{expected shift in time}, that is $\EE[\phi(X_{(i+j)\dt}\,\vert\,X_i]=A_{j\Delta t}^* \phi(X_i)$. In another words, the transfer or Koopman operator for a time-lag $j\dt$ (which we can estimate from the trajectory data) act on a data matrix of time-ordered feature maps (in expectation) as the shift by \(j\) column indices, which we can express by the multiplication with a unit $j$-th diagonal matrix $\textsc{D}_j$ as
\begin{equation}\label{eq:main_idea}
\EE\Big[A_{j\dt}^*\;[\phi(X_{\dt})\,\vert\,\phi(X_{2\dt}),\vert\cdots \vert \phi(X_{\ell\dt})\,\vert\,0\,\vert\cdots\vert\,0\,]\Big] = \EE\,[\phi(X_{\dt})\,\vert\,\phi(X_{2\dt}),\vert\cdots \vert \phi(X_{n\dt})]\, \textsc{D}_{-j}.  
\end{equation}
As a result, different finite polynomial, Chebyshev or trigonometric expansions of \(\anfun\) induce linear operators whose action on the feature maps is represented by banded Toeplitz matrices with diagonals indexed by the time lags \(j\) and weights given by the coefficients of the expansion. 

To summarize the fundamental idea of this paper: \textbf{the functional calculus for the generator of dynamics transforms into structured linear algebra on transfer operator semigroup}.
This approach naturally allows to build data-driven estimators based on time-lag cross covariance operators that can be empirically estimated. Notably, if the Toeplitz symbol is carefully chosen, our approach incorporates prior knowledge on the system and its operator, such as self-adjointness or skew-symmetry, resulting in better eigenvalue estimation. For instance, for deterministic systems the transfer operator is unitary and the eigenvalues lie on the unit circle and our approach preserves this property. Moreover, we show that the derived estimators can be implemented with efficient Toepliz matrix computations, leveraging primal or dual algorithms from statistical learning. Finally, we show that the estimates are statistically consistent.

The paper is organized as follows. Section \ref{sec:background} provides background material and defines linear evolution operators and their key properties, notably their spectral decomposition. Section \ref{sec:problem} introduces the main idea behind Toeplitz-based spectral methods and the statistical learning setting. Section \ref{sec:methods} presents our estimators, the corresponding primal and dual algorithms and comments on computational considerations. Section \ref{sec:consistency} establishes the statistical consistency of the proposed general approach, presents flagship choices of spectral filters and discusses their use. Finally, Section \ref{sec:exps} presents numerical experiments on conceptual dynamical system illustrating the potential of the proposed method in comparison to common data-driven approaches. In particular, the experiments suggest that our method is promising, that specific filters can improve not only forecasting, but also yield better estimated eigenvalues, and spectral measures for chaotic systems.

%our method combines multiple TOs at different time-lags through a single matrix product between a Toeplitz matrix and the kernel embedding,

%In the future our approach could unlock methods to find invariant subspaces which are not dominant such those arising on mixed deterministic-stochastic systems, which reveal important physical properties.
%Furthermore, in the future, our algorithms' complexity could be further reduced while preserving accuracy by utilizing standard scaling techniques such as random Fourier features or other randomized methods that have already been used for the transfer operator \citep{meanti,Turri}.

\section{Linear operator perspective to stable dynamical systems}\label{sec:background}% (1pg)}

%\VK{[Content: (1) Background:  Markov semigroups and generators, stability via invariant measure, geometric ergodicity, spectral decomposition, sectorial generators, link to SDEs, examples; (2) Problem: learning from data, RKHS, restrictions of generator vs resolvent to RKHS, risk, difficulty of integral]}

Many temporal phenomena in science and engineering evolve as deterministic or stochastic processes
$X = (X_t)_{t \ge 0}$ taking values in a state space $\mathcal X \subset \mathbb{R}^d$.
We focus on continuous-time Markov processes with continuous paths, including It\^o diffusions
and their deterministic limits.

%Many temporal phenomena in science and engineering evolve as deterministic or stochastic processes $X=(X_t)_{t\ge0}$ in a state space $\spX \subset \mathbb{R}^d$. We focus on continuous-time Markov processes with continuous paths, which include Itô diffusions, reflected or time-changed Brownian motions, and, in the zero-diffusion limit, deterministic flows, among others.  

The dynamics of $X$ is described by the transition densities $(p_t)_{t\ge0}$, so that 
%
%\begin{equation}\label{Eq: semigroup def}
$\PP(X_t\in E|X_0=x)=\textstyle{\int_E} p_t(x,y)dy$,  %\end{equation}
and the associated \emph{transfer operators} $(\TOp_t)_{t\ge0}$ acting on observables $f:\spX \to \R$ via
%and \textit{transfer operators} (TO) $(A_t)_{t\in\R_+}$ such that for all $t\!\in\!\R_+$, $E\in\mathcal B(\spX)$, $x\in\spX$ and measurable function $f\colon\!\spX\!\to\!\R$,
\begin{equation}
\label{Eq: transfer operator def}
\TOp_tf=\textstyle{\int_\spX} f(y)p_t(\cdot,y)dy=\EE\big[f(X_t)\,|\,X_0=\cdot\big].\end{equation}

Time homogeneity of the Markov process implies that the transition densities satisfy
the Chapman--Kolmogorov equations
$p_{t+s}(x,y)=\int_{\spX} p_t(x,z)p_s(z,y)\,dz$ for all $t,s\ge0$,
which is equivalent to the semigroup property
$\TOp_{t+s}=\TOp_t\TOp_s$.
This property expresses the autonomy of the dynamics and will later be interpreted
as time equivariance of the evolution.

Transfer operators are key to understanding the dynamics of $X$. We study them on $\Lii$, the space of square-integrable functions with respect to an \emph{invariant measure} $\pi$, which satisfies $\TOp_t^* \pi = \pi$ for all $t\ge0$. We assume that $X$ satisfies: (1) \emph{long-term stability}, i.e., convergence in distribution to $\pi$ from any initial state in its support, and (2) a \emph{$\beta$-mixing} property, meaning that the correlation between $f(X_t)$ and $f(X_0)$ decay exponentially in $t$, for all $f \in \Lii$.

The semigroup property of $(\TOp_t)_{t\ge0}$ expresses a fundamental symmetry of autonomous dynamics, namely \emph{time equivariance}. At the level of trajectories, this corresponds to the action of the time-shift operators $(\tau_s X)_t := X_{t+s}$, $s\ge0$, which commute with the forward evolution: shifting the time origin and then propagating observables yields the same result as propagating first and then shifting. Time equivariance is thus a structural property of the dynamics induced by autonomy and holds for both deterministic and stochastic systems, without requiring invertibility or time reversibility. In contrast, \emph{time invariance} refers to stationarity at equilibrium: if $X_0\sim\pi$, then the law of $X_t$ is independent of $t$, equivalently $\TOp_t^*\pi=\pi$ for all $t\ge0$, or $\EE_\pi[f(X_t)]=\EE_\pi[f(X_0)]$ for all $f\in\Lii$. Time equivariance concerns the commutation of evolution with time shifts, while time invariance characterizes equilibrium distributions; the two notions are distinct and coexist in stable stochastic systems, whereas deterministic systems on simple attractors exhibit equivariant dynamics without probabilistic stationarity beyond the invariant measure supported on the attractor.

The process $X$ is also characterized by its \emph{infinitesimal generator} $L$, defined on its domain $\dom(L) \subset \Lii$ by
\[
L f := \lim_{t\to0^+} \frac{\TOp_t f - f}{t}, \quad f \in \mathcal D(L),
\]
with $L$ closed. Under the above assumptions, $(\TOp_t)_{t \geq 0}$ forms a strongly continuous contraction semigroup on $\Lii$, so that, for all $t \ge 0$, $\TOp_t = e^{t L}$ \citep[see e.g.][]{engel2000}.

\textbf{Spectral decomposition.} When continuous for some $\mu \in \mathbb{C}$, the operator $\Rx= (\mu \Id - \generator)^{{-}1}$ is the \textit{resolvent} of $\generator$, and $\rho(\generator)=\big\{\mu \in\mathbb C\,|\, \mu\Id-\generator\, \,\text{is bijective and}\,\Rx  \,\text{is continuous}\big\}$ is called the \textit{resolvent set}. For a sectorial operator, the resolvent is uniformly bounded in $\mu$ outside a sector containing the spectrum. The spectral decomposition of $\generator$ is
%can be written as
%
\begin{equation}\label{Eq: spectral dec generator}
\generator = \textstyle{\sum_{i=0}^{\infty}}\geval_i\,P_i + \int_{\Spec_c(\generator)}\lambda dE(\lambda) 
\end{equation}
where $(\geval_i)_{i \in \mathbb{N}_0}\subset\C$ are the eigenvalues, $P_i$'s are the corresponding spectral projectors, $\Spec_c$ denotes the continuous spectrum and $E$ is the spectral measure.  In the following, to ease the presentation, we assume that all eigenvalues are simple, that is, $P_i=\gefun_i\otimes \lgefun_i$, where $f_i, g_i \in\Lii$ are the corresponding left and right eigenfunctions. 

\textbf{Link with SDEs.} \emph{Itô diffusion processes} %(studied in \citep{Kostic2024diffusion}) 
are a key example of Markov processes, governed by stochastic differential equations (SDEs) of the form
\begin{equation}\label{Eq: SDE} dX_t = a(X_t) dt + b(X_t) dW_t, \quad X_0 = x, \end{equation} 
where $x \in \spX$, $W = (W_t^1, \dots, W_t^p)_{t \in \R^{+}}$ is a standard $p$-dimensional Brownian motion, the drift $a : \spX \to \R^d$ and diffusion $b : \spX \to \R^{d \times p}$ are globally Lipschitz and sub-linear. This ensures a unique solution $X = (X_t)_{t \geqslant 0}$ in $(\spX, \mathcal B(\spX))$. 
%SDEs like \eqref{Eq: SDE} include Langevin dynamics %(see Example \ref{ex: Langevin}) and Ornstein-Uhlenbeck processes. 
%(see Example \label{ex: OU})  with broad applications in science and engineering
The generator $\generator$ associated with \eqref{Eq: SDE} is a second-order differential operator, defined on 
the Sobolev space $\Wii{=}\{f\in\Lii\;\vert\; \norm{f}_{\Liishort}{+}\norm{\nabla f}_{\Liishort}{<}\infty\}$, as
\begin{equation}\label{Eq: def generator coeff} \generator f(x) {=} \nabla f(x)^\top a(x) + \tfrac{1}{2}\mathrm{Tr}\big[b(x)^\top (\nabla^2 f(x)) b(x)\big],\quad f \in \Lii,\, x \in \spX
 \end{equation} 
where $\nabla^2 f {=} (\partial_{ij}^2 f)_{i,j \in [d]}$ is the Hessian of $f$, and $[d]:=\{1,\dots,d\}$. 

In virtue of \eqref{Eq: transfer operator def}, every function that is almost everywhere constant w.r.t. $\im$ is both a left and right eigenfunction of $A_t$, with eigenvalue one, implying that $\generator\one_\pi=0$. Thus, we can focus on the nontrivial part of the generator's spectra. To this end, we let  
$\Liio=\{f\in\Lii\,\vert\, \EE_{X\sim\im}[f(X)]=0\}$ be the subspace of $\Liishort$ orthogonal to $\one_\im$, and $\generator_0\colon \Wiio  \to\Liio$ the $\Liioshort$-generator, where $\Wiio{=}\{f\in\Wii\;\vert\; \EE_{X\sim\im}[f(X)]=0\}$. We also define the deflated operator $\overline{A}_t:=e^{t \generator_0} = J_\im A_t$,  %= J_\im A_t J_\im$, 
where $J_\im\colon\Lii\to\Lii$ is the orthogonal projector onto $\Liio$, that is $J_\im=I-\one_\im\otimes\one_\im$. 

The spectral decomposition of $\generator$ and $\generator_0$ allows one to solve the SDE \eqref{Eq: SDE}. In particular, assuming for simplicity the absence of the continuous spectrum and non-defective discrete spectrum,  
%with the above notation 
we have that
\begin{equation}
\label{eq:solutions}
\EE[f(X_t)\,\vert\,X_0\!=\!x] {=} \EE_{X\sim\im}[f(X)]+ \textstyle{\sum_{i\in\N}} \,e^{\eval_i t}\,\scalarp{\lefun_i,f}_{\Liioshort}\gefun_i(x),\;f\in\Lii.
\end{equation}
We next discuss specific examples of Markov processes which are covered by the methodology presented in this paper; see \citep{lasota1994} for more information.

\medskip

\begin{example}[Overdamped Langevin]\label{ex: Langevin} 
The \emph{overdamped Langevin} dynamics of a particle in a potential $V:\R^d\rightarrow\R$ satisfies \eqref{Eq: SDE} with $a=-\gamma^{-1}\nabla V$ and $b\equiv \sqrt{2\betaLang} I_{d \times d}$, where $\gamma>0$ is a friction coefficient, $T$ is the temperature and $k_b$ is the Boltzmann constant. The invariant measure is the \emph{Boltzmann distribution} $\pi(dx) \propto e^{-V(x)/(k_b T)} dx$.
\end{example}
%The IG $\generator$ for $f \in W^{2,2}$ is $\generator f=-\gamma^{-1}\nabla V^\top \nabla f+\betaLang\Delta f$, for $f\in \Wii$. 
%Since $\int (-\generator f)g\,d\pi
%=-\int \Big[\nabla\Big(\betaLang\nabla f(x)\frac{e^{-{\betaLang}^{-1} V(x)}}{Z}\Big)\Big]g(x)dx
%= \betaLang\int \nabla f^\top \nabla g\,d\pi=\int f(-\generator g)\,d\pi$, the generator $\generator$ is self-adjoint. 
%In dissipative systems, the IG $\generator$ is sectorial, with its spectrum usually in the left half-plane. 
%For confining potentials, the spectrum is discrete, featuring 0 as the largest eigenvalue, which corresponds to the distribution $\im$.

\medskip

\begin{example}[Ornstein-Uhlenbeck process]\label{ex:OU} This process is governed by the SDE \eqref{Eq: SDE} with $a(x)=Ax$ and $b\equiv B$, where $A \in \R^{d\times d}$ and $B \in \R^{d\times d}$ are the drift and diffusion matrices. This models systems like the Vasicek interest rate
%, damped harmonic oscillators, 
and neural dynamics, where fluctuations return to equilibrium.
If the real parts of $A$'s eigenvalues are negative, the process has an invariant Gaussian distribution with covariance $\Sigma_\infty$ satisfying Lyapunov's equation: $A\Sigma_\infty + \Sigma_\infty A^\top = -BB^\top$. %Its IG 
%for $f \in L^2$ is $\generator f(x)=\nabla f(x)^\top Ax + \frac{1}{2}\mathrm{Tr}[B^\top (\nabla^2 f(x)) B]$. The IG 
%has a discrete spectrum with eigenvalues related to the drift matrix $A$, where 0 corresponds to the distribution $\im$, and the rest are negative, reflecting relaxation rates.
\end{example}

In the context of an Itô diffusion \eqref{Eq: SDE} with non-degenerate noise, i.e. $bb^\top$ invertible
%positive definite 
a.s. w.r.t $\im$, the generator $L$ is an elliptic operator which typically has compact resolvent under confining conditions. This implies a \textit{purely discrete spectrum} (eigenvalues with finite multiplicity) accumulating at $-\infty$. However, if the diffusion is removed (i.e., $b=0$) the dynamics becomes deterministic, and the corresponding generator $L = a \cdot\nabla$, again defined on the $\Lii$ space, where $\im$ is now supported on an attractor, undergoes a radical spectral transition.

\medskip
 
\begin{example}[Duffing oscillator]\label{ex: Duffing}
This is a classical example of deterministic system $(X_t)_{t\ge0}$ in $\mathbb{R}^2$, written as the ODE
$\dot{x} = y$, $\dot{y} = -\delta y - \alpha x - \beta x^3 + \gamma \cos(\omega t)$, where $(x(t),y(t)) = X_t$ and $\delta,\alpha,\beta,\gamma,\omega$ are real parameters. For suitable choices, the system has a simple attractor, e.g., a limit cycle or stable fixed point, on which the invariant distribution $\pi$ is supported. While the system in non-autonomous in the coordinates $(x(t),y(t))$, by including the forcing term $\omega(t)=\omega t$ one can build a strongly continuous Markov semigroup and define its generator.
%The infinitesimal generator $\generator$ is give by  \[L f = F \cdot \nabla f, \qquad F(X_t) = (y_t, -\delta y_t - \alpha x_t - \beta x_t^3 + \gamma \cos(\omega t)).\]
\end{example}

% \begin{example}[Lorentz63]\label{ex: Lorentz}
% The \emph{Lorenz '63 system} is a deterministic system $X_t=(x_t,y_t,z_t) \in \mathbb{R}^3$, defined by the ODE: $\dot{x}_t = \sigma(y_t - x_t)$, $\dot{y}_t = x_t(\rho - z_t) - y_t$, $\dot{z}_t = x_t y_t - \beta z_t$. For standard parameters, it has a strange attractor with an invariant measure $\pi$. Transfer operators $\TOp_t$ and their generator $L$ act on $\Lii(\spX)$ via $\generator f = F\cdot \nabla f$, with $F(X_t)$ given by the right-hand side above.   
% \end{example}

%In practice, the above examples may be further observed with additive noise, $Y_t = X_t + \eta_t$, and the dynamics become stochastic, with a corresponding generator and invariant measure $\pi$ that typically smooths the attractor.

\textbf{Deterministic vs. stochastic systems.} For every deterministic dynamical system, the operator $\generator$ is skew-adjoint, implying purely imaginary spectrum and strongly continuous unitary semigroup of the Koopman operators $A_t=e^{t L}$ that are time-reversal equivariant at equilibrium, that is $A_{-t}=A_{t}^*=A_t^{-1}$. For a \em{simple attractor}, e.g., a stable limit cycle as often arising in the context of Example \ref{ex: Duffing}, the spectrum of $\generator_0$ typically remains discrete, as the dynamics are regular and the invariant measure is smooth along the attractor. In contrast, for a \em{chaotic attractor} satisfying strong mixing properties, the $\Lii$-spectrum of $L$ becomes \textit{continuous} on the imaginary axis, typically filling the whole axis.  More precisely, what survives as discrete objects are not classical eigenvalues but \textit{Ruelle–Pollicott resonances}, which are singularities of the meromorphically extended resolvent and lie in the left half-plane; they describe decay rates of correlations but do not correspond to $\Lii$-eigenfunctions \citep{ueda1980steady,guckenheimer2013nonlinear,mezic2020spectrum}. Thus, the presence or absence of noise qualitatively transforms the spectral theory.

%\textbf{Dissipative stochastic systems.} 
In contrast to the above, the infinitesimal generator $L$ of a dissipative stochastic systems at a stationary distribution is not skew-adjoint, but it does possesses a particularly regular spectral structure. In particular, for time-reversal invariant systems, see Example \ref{ex: Langevin}, $\generator$ is self-adjoint and, so, its spectrum lies entirely on the non-positive real axis; it features a simple eigenvalue at zero, corresponding to the steady state, and a spectral gap that governs the slowest exponential relaxation rate. On the other hand, generators on non-equilibrium steady states, arising from external forcing, are no longer time reversible, typically leading to complex eigenvalues that signal oscillatory relaxation modes. Still, these remain bounded by time-scales and the associated functional-analytic framework generally persist. This is known as the class of \textbf{sectorial generators} that generate strongly continuous semigroups, analytic in a sector of the complex plane defined by growth conditions in an angular region, i.e., $\generator$ is a (stable) sectorial operator with angle $\theta \in [0, \pi/2)$, 
\begin{equation}\label{eq:sectorial}
\fov(\generator)\subseteq \C_{\theta}^{-}:=\{z\in\C\;\vert\; \Re(z)\leq0\;\;\wedge\;\;\abs{\Im(z)}\leq - \Re(z) \tan(\theta)\},
\end{equation}
where $\fov(\generator)$ denotes the \emph{numerical range} of $\generator$. This class covers all time-reversal processes (self-adjoint $\generator$), but also important non-time-reversal processes, such as Advection-Diffusion and underdamped Langevin~\citep{kloeden1992}. 

{\textbf{Analytical functions of the generator.} As noted above,  $\generator$ is a (typically unbounded) differential operator. While classical numerical algorithms, such as finite-element methods (FEM), for computing its spectral decomposition have been a cornerstone in scientific computing, they typically suffer form the curse of dimensionality. Hence, the methods of choice in high-dimensional settings are statistical in nature. While an optimal estimation of the spectral components of $L$ in the general setting remains an open problem, recent advances have successfully resolved the dominant spectrum of dissipative stochastic systems \citep{kostic2025laplace}. The core idea is to approximate the generator's resolvent by a polynomial of $A_{\dt}=e^{\dt L}$, where $\dt$ is the time-lag at which the system is observed. Since analytical transforms change only eigenvalues while preserving the eigenfunctions, such a polynomial can be associated to a Toeplitz matrix encoding weighted time-lags used to compute the empirical spectral decomposition. In the following, we show how this idea can be exploited to build empirical estimators for diverse analytical transforms across different types of dynamics, including the above examples, so that empirical estimators preserve key spectral properties of the generator, as summarized in \cref{tab:summary}.}

%incurring the $\Liishort$-error ... \VK{[TBW]}.
%

\begin{table*}[t!]
\centering
\renewcommand{\arraystretch}{1.2} % Adjust this value to increase/decrease vertical spacing
\resizebox{\textwidth}{!}{%
\begin{tabular}{l|cccc>{\columncolor{yellow!20}}c}
Operator $\anfun(L)$ & Toeplitz Symbol & Toeplitz Matrix &  Dominant Spectrum  & Generator Spectrum & Type of Dynamics \\ 
\midrule
\hline
$A_t = e^{\dt L}$ & $\toepsymb(z) = z$ & $\weight_1 = 1$ & right-most & general & general stable  \\ \hline
$A_t = e^{\dt L}$ & $\toepsymb(z) = (z+z^{-1})/2$ & $\weight_{\pm 1} = 1/2$ & largest modulus & real negative & \begin{tabular}[c]{@{}c@{}}
stable stochastic \\ time-reversal invariant \end{tabular}\\ 
\hline
$\sinh{\dt L}$ & $\toepsymb(z) = (z-z^{-1})/2 $ & $\weight_{\pm 1} = \pm 1/2$ &  \begin{tabular}[c]{@{}c@{}}
largest modulus \\ (highest frequencies) \end{tabular} & discrete imaginary & \begin{tabular}[c]{@{}c@{}}
deterministic on \\ simple attractor \end{tabular} \\ \hline
$\cosh{\dt L}$ & $\toepsymb(z) = (z+z^{-1})/2$ & $\weight_{\pm 1} = 1/2$& \begin{tabular}[c]{@{}c@{}}
smallest modulus \\ (lowest frequencies) \end{tabular} & discrete imaginary & \begin{tabular}[c]{@{}c@{}}
deterministic on \\ simple attractor \end{tabular}  \\ \hline

$(e^{\mu}-e^{\Delta t L})^{-1}$ & $\toepsymb(z) = [e^{\mu}-e^{z}]^{-1}$ & \eqref{eq:toep_res_to} & closest to $\mu$ & sectorial or imaginary & general stable \\ \hline

$(\mu-L)^{-1}$ & $\toepsymb(z) = [\mu-{\rm Ln}\, z]^{-1}\quad$ & \eqref{eq:trapezoid} & closest to $\mu$ & sectorial or imaginary & general stable \\ \hline

$(\mu-L)^{-1}$, $\mu>0$ & $\toepsymb(z) = [\mu-{\rm Ln}\frac{z+z^{-1}}{2}]^{-1}$ & \eqref{eq:toep_res_sym} & largest & real & \begin{tabular}[c]{@{}c@{}}
stable stochastic \\ time-reversal invariant \end{tabular} \\ \hline

$P_{(\omega_{\min},\omega_{\max})}\,L_0^{-1}$ & $\toepsymb(z) = \frac{\one_{\{|{\rm Arg}(z)|\in[\omega_{\min}, \omega_{\max}]\}}}{{\rm Ln}\,z}$ & \eqref{eq:toep_inv}& \begin{tabular}[c]{@{}c@{}}
frequencies in the range \\
$(\omega_{\min}/2\pi,\omega_{\max}/2\pi)$ \end{tabular} & discrete imaginary & \begin{tabular}[c]{@{}c@{}}
deterministic on \\ simple attractor \end{tabular} \\ \hline
General $\anfun(\generator)$ & \begin{tabular}[c]{@{}c@{}}
Trigonometric and Chebyshev \\ spectral filters \end{tabular} & \eqref{eq:toep_trig} and \eqref{eq:toep_cheb} 
  & \begin{tabular}[c]{@{}c@{}}
largest $\anfun$-values \\ (largest filtered frequencies) \end{tabular} & imaginary & general deterministic \\ \hline
\end{tabular}
}
\caption{
%Summary of the appropriate 
Instances of Toeplitz based estimators for different types of dynamical systems. Explicit coefficients of the Toeplitz matrices used in Algorithms \ref{alg:primal} and \ref{alg:dual} are found in the referenced equations, $\mu$ is a parameter with non-negative real part, and $P_{(\theta_{\min},\theta_{\max})}$ is the spectral projector onto the frequency band. In all cases, we can learn $\generator_0$ instead of $\generator$ simply by centering the features in the hypothesis space $\RKHS$.
}
\label{tab:summary}
\end{table*}

\section{Statistical learning of analytic transforms of the generator}\label{sec:problem}

Consider a (stochastic) differential equation \eqref{Eq: SDE}, with an invariant measure $\pi$ and  generator $\generator\colon \Wii\to\Lii$. Recalling that the generator's spectra lies in the complex left half-plane, $\Spec(\generator)\subseteq \C^{-}$, and, so, the transfer operators' spectra lie in the unit disc, $\Spec(A_t)\subseteq \mathbb{D}=\{z\in \mathbb{C}\,\vert\, |z|\leq 1\}$, we consider a Toepltiz symbol $\toepsymb\colon \mathbb{T} \to \mathbb{C}$ on  the unit circle $\mathbb{T}\subset\mathbb{C}$,
\begin{equation}\label{eq:toeplitz_symbol}
\textstyle{\toepsymb(z)= \sum_{j\in\mathbb{Z}} \,\weight_j\, z^j,}    
\end{equation}
which we extend to the unit disk by $\toepsymb(z)= \weight_0+\sum_{j\in\N}(\weight_j z^j + \weight_{-j} \bar{z}^j)$, for all $z \in \mathbb{D}$. Furthermore, for every $\ell \in \mathbb{N}$, we define $\toepsymb_\ell(z):= \weight_0+\sum_{j\in[\ell]}(\weight_j z^j + \weight_{-j} \bar{z}^j)$ the $\ell$-{\em truncated symbol}, noting that error bounds $|\toepsymb-\toepsymb_\ell|$ have been studied for rich classes of symbols \citep[see, e.g.,][]{bottcher2005, gray2006toeplitz}.  

Given a time-lag $\dt>0$, the task is to estimate the operator $\anfun(L):=\toepsymb(A_{\dt})$ from a trajectory of equally $\dt$-time spaced data $\Data=\{x_i\}_{i=1}^{n}$ from a stationary distribution $\pi$. Motivating examples are summarized in \cref{tab:summary} and include the classical case of Koopman/Transfer operator estimation for $\toepsymb$ being the identity.

To present the new class of methods, we follow the formalism of learning in Reproducing Kernel Hilbert Spaces (RKHS), well established in the context of transfer operator regression \citep{Kostic2022,kostic2023sharp}. To that end, let $\RKHS$ be an RKHS with kernel $k:\spX\times\spX \to \R$, and $\fH:\spX \to \RKHS$ be a {\em feature map} such that $k(x,x^\prime) = \scalarp{\fH(x), \fH(x^\prime)}$ for all $x, x^\prime \in \spX$. We assume that $\RKHS \subset \Lii$, enabling us to approximate $\anfun(\generator):\Wii \to \Lii$ with an operator $\Estim:\RKHS \to \RKHS$. Although $\RKHS$ is a subset of $\Lii$, they have different metric structures, so for $f, g \in \RKHS$, $\scalarp{f, g}_{\RKHS} \neq \scalarp{f,g}_{\Liishort}$. To resolve this, we introduce the \emph{injection operator} $\TS:\RKHS \to \Lii$, which maps each $f \in \RKHS$ to its pointwise equivalent in $\Lii$ with the appropriate $\Liishort$ norm and note that a direct calculation shows that its adjoint $\TS^*\colon\Lii\to\RKHS$ acts as the Bochner integral $\TS^*f = \EE_{X\sim\im}[f(X)\phi(X)]\in\RKHS$, $f\in\Lii$.

In machine learning, two typical choices for the space $\RKHS$ and the associated kernel $k$ arise. The first one, fundamental in learning theory, are  \textit{universal kernels}, for which $\RKHS$ is dense in $\Lii$. A common example is the Gaussian kernel, $k(x,x^\prime) = \exp\{\frac{\|x-x'\|^2}{2\sigma }\}$; see, e.g., \citep{Steinwart2008} for more examples. Moreover, if the kernel is bounded the injection operator $\TS$ is Hilbert-Schmidt, allowing one to efficiently learn bounded operators on $\Liishort$ via finite rank approximations \citep{Kostic2022}. The second choice are finite-dimensional RKHS, obtained by \textit{dictionaries of functions} $(z_i)_{i=1}^m\subseteq\Lii$, corresponding to the kernel  $k(x,x')=z(x)^\top z(x')$ and $\phi(x)= z(x)^\top z(\cdot)$, where  $z(x)=[z_1(x)\vert\,\cdots\vert z_m(x)]^\top$. In this case, one  readily sees that $\RKHS$ is isometric to $\mathbb{R}^m$ and operators on $\RKHS$ to $m \times m$ matrices. 
%\VK{[TBW, norms, isometric-isomorphism with matrices, covariance matrices, can be invertible, to properly learn dictionary needs to span the L2 space when m grows] }.

With this setting in mind, in order to introduce the learning problem we restrict to the class of functions $\anfun$ that are bounded on the spectrum of $L$, allowing us to formulate learning in $\Liio$ as a linear inverse problem
\begin{equation}\label{eq:inv_prob}
\text{Find } \quad\Estim\colon\RKHS\to\RKHS\quad \text{ s.t. }\quad \anfun(\generator)\TJ\TS=\TJ\TS\Estim,
\end{equation}
whose solution is $\HKoop=(\TS^*\TJ\TS)^\dagger \TS^*\TJ\anfun(\generator)\TJ\TS$. Thus, recalling the definition of Toeplitz symbol, the injection operator and its adjoint, the following characterization easily follows. 
\begin{proposition}\label{prop:inv_prob_solution}
Given $\dt>0$, let $\Cxy{j}\colon\RKHS\to\RKHS$ be the $j\dt$-time-lagged cross-covariance operator
\begin{equation}\label{eq:cross_cov}
\Cxy{j}=\EE_{X_0\sim\im}\,[\phi(X_0)-\EE_{X\sim\im}\phi(X)]\otimes[\phi(X_{j\dt})- \EE_{X\sim\im}\phi(X)],\; j\in\N_0.   
\end{equation}
Then the solution of the inverse problem \eqref{eq:inv_prob} is given by $\HKoop = \Cxy{0}^\dagger \Yx$ , where
\begin{equation}\label{eq:inv_prob_solution}
\Yx = \textstyle{\weight_0\Cxy{0}+\sum_{j\in\N}[\,\weight_j\Cxy{j}+\weight_{-j}\Cxy{j}^*\,],   }
\end{equation}
represents the $\weight$-weighted time-lagged cross-covariance operator. 
\end{proposition}
\begin{proof}
Multiply equation \eqref{eq:inv_prob} by $\TS^*$ from the right and note that $C_0= \TS^* \TJ\TS$.
\end{proof}

Finally, recalling \eqref{eq:main_idea}, we observe that when estimating  \eqref{eq:inv_prob_solution} from trajectory data we hit a limit whenever Toeplitz coefficients $a_{j}$ are nonzero for $|j|$ larger than the trajectory length $n$. In such cases, the best we can do is to approximate $\anfun(L)=\toepsymb(A_{\dt})$ by a $\ell$-truncation $\toepsymb_\ell$ of $\toepsymb$, that is by $\anfun_\ell(\generator):=\toepsymb_\ell(A_{\dt})$. Then, the additional operator approximation error can be controlled  by the the error on the spectrum for normal operators, while in general we can use the fundamental spectral-set property of the numerical rage
\begin{theorem}[Crouzeix's Theorem, see \cite{crouzeix2017numerical}] 
\label{thm:crouzeix}
Let $A$ be a bounded linear operator on a Hilbert space $\mathbb{H}$, and let $f$ be a function analytic on $\fov(A)$. Luckily \[ \| f(A) \| \leq (1+\sqrt{2}) \max_{z \in \cl\!\fov(A)} |f(z)|, \] where $\cl\!\fov(A)$ denotes the closure of numerical range of $A$. 
\end{theorem}
More precisely, assuming that $\anfun$ and $\toepsymb_\ell$ are analytic on $\fov(L)$, we can focus on learning $\toepsymb_\ell$ knowing that the approximation error $\norm{\anfun(\generator)-\anfun_\ell(\generator)}_{\Lii\to\Lii}$ is bounded by 
\begin{equation}\label{eq:app_error_fun}
{\rm err}_\ell(\anfun):=
\begin{cases}
\sup_{\omega\in \Spec(\generator)}|\anfun(\omega)-\anfun_\ell(\omega)|, & \generator\generator^*=\generator^*\generator, \\
(1{+}\sqrt{2})\,\sup_{\omega\in \fov(\generator)}|\anfun(\omega)-\anfun_\ell(\omega)|, & \text{ otherwise.}   
\end{cases} 
\end{equation}

We conclude this section by making the link between the inverse problem \eqref{eq:inv_prob} and the risk-based operator learning for three particular cases, (i) $\anfun$ with polynomial expansions for general $\generator$, (ii) general $\anfun$ for self-adjoint $\generator$ (time-reversal-invariant dynamical systems) and (iii) general $\anfun$ for skew-adjoint $\generator$ (deterministic dynamical systems). To that end, we observe that $\LKoop_{t}^*f = \LKoop_t f = \EE[f(X_t)\,\vert\,X_0=\cdot]$ holds for self-adjoint generators, while $\LKoop_{t}^*f= \EE[f(X_{-t})\,\vert\,X_0=\cdot]$ is true for skew-adjoint ones. Therefore, in all three cases above the risk can be reformulated as the mean square error (MSE) of predicting the target feature  
\begin{equation}\label{eq:laplace_embedding}
\IfH(X_0)=
\begin{cases}
%\displaystyle
{\sum_{j\in\N_0}}\,a_j\fH(X_{j\dt}), & a_j = 0 \text{ for all } j<0,\\
%\displaystyle
{\sum_{j\in\mathbb{Z}}}\,a_j\fH(X_{|j|\dt}), & \generator^*=\generator,\\
%\displaystyle
{\sum_{j\in\mathbb{Z}}}\,a_j\fH(X_{j\dt}), & \generator^*=-\generator,
\end{cases}
\end{equation} 
by $\Estim^*\phi(X_0)$ in the stationary distribution $\im$, that is  
\begin{equation}\label{eq:true_risk}
%\min_{\Estim\colon\RKHS\to\RKHS}
\Risk(\Estim)= 
\EE_{X_0\sim \im}\norm{\IfH(X_0) -\Estim^*\fH(X_0)}_{\RKHS}^2. 
\end{equation}
So, as essentially shown in \citep{Kostic2022}, we have that the inverse problem \eqref{eq:inv_prob} solution is the minimizer of the excess risk $\ExRisk(\Estim)\!=\!\Risk(\Estim) \!-\! \min_{\EEstim}\Risk(\EEstim) \!=\! \norm{\anfun(\generator)\TJ\TS\!-\!\TJ\TS\EEstim}^2_{\HS{\RKHS,\Liishort}}$. Furthermore, if $\RKHS$ is dense in $\Lii$ and the injection operator is Hilbert-Schmidt, then one can find arbitrarily good finite-rank approximations of {$\anfun(\generator)\TJ\TS$}. Otherwise, if we learn in non-universal RKHS, there is a representation error $\norm{(I-P_{\RKHS})\anfun(\generator)}$, $P_\RKHS$ being the orthogonal projection on the $\RKHS$ in $\Liio$, which affects the estimation, see e.g. \citep{Kostic_2023_learning,Kostic2024forecasting}. 

%As reported in \citep{kostic2023sharp} and \citep{Kostic2024diffusion}, a significant challenge in spectral decomposition estimation arises. Specifically, as the estimator's rank increases, the metric distorsion between $\RKHS$ and $\Lii$ negatively impact learning. 
%The main bottleneck of the risk functional \eqref{eq:true_risk} is clearly the problem of computing the integral in \eqref{eq:laplace_embedding}, which prevents the use of standard operator regression approach.

%The bottleneck in the risk functional \eqref{eq:true_risk} is computing the integral in \eqref{eq:laplace_embedding}, which hinders the use of standard operator regression methods. %Although \citep{Kostic2024diffusion} proposes an alternative for learning eigenfunctions of the generator through an energy-based risk functional, this approach is limited to diffusions with a Dirichlet form and self-adjoint generators. 
%We present in the next section a novel general approach that clears this difficulty.% leveraging \eqref{eq:resolvent}.

%While in \citep{Kostic2024diffusion} an alternative to learning the resolvent is proposed by adapting the norm when $\generator=\generator^*$ has a known, or estimated, Dirichlet form, the problem remains unsolved for more general processes. 

%In the following sections, we will address this problem by generalizing the transfer operator analysis from \citep{kostic2023sharp} to the semigroup case using the IG's resolvent, with key results summarized next.

%\karim{\cite{hou23c} covers trajectories beta-mixing condition}

\section{Toeplitz based empirical estimators}\label{sec:methods} % (1.5pg)}

In this section we assume to have access to a dataset $\Data=(x_{i})_{i=1}^{n}$ obtained by sampling the process $(X_t)_{t\geq0}$ at some sampling frequency $1/\dt$ for $\dt>0$ being typically small in order to observe all the relevant time-scales and oscillatory frequencies of the process. In the context of deterministic systems, this introduces the fundamental limit of Nyquist frequency $\pi/\dt$ for any empirical estimator. While given such trajectory data $\Data$ one can clearly estimate $\Cxy{j}$ given by \eqref{eq:cross_cov} by its empirical version 
% \begin{equation}\label{eq:emp_cross_cov}
% \textstyle{\ECxy{j}=\frac{1}{n-j-1}\sum_{i=1}^{n-j}\,\left[\phi(x_i)-\frac{1}{n-j}\sum_{i=1}^{n-j}\phi(x_i)\right]\otimes\left[\phi(x_{j+i})- \frac{1}{n-j}\sum_{i=1}^{n-j}\phi(x_{j+i})\right]},\;\; 0\leq j\leq n-2,   \end{equation}
% \begin{equation}\label{eq:emp_cross_cov}
% \textstyle{\ECxy{j}=\frac{1}{n-j}\sum_{i=1}^{n-j}\,\phi(x_i)\otimes\phi(x_{j+i}) - \big[\frac{1}{n}\sum_{i\in[n]}\phi(x_{i})\big]\otimes \big[\frac{1}{n}\sum_{i\in[n]}\phi(x_i)\big]},\;\; 0\leq j\leq n-1,   \end{equation}
\begin{equation}\label{eq:emp_cross_cov}
\textstyle{\ECxy{j}=\frac{1}{n-j}\sum_{i=1}^{n-j}\,\left[\phi(x_i)-\frac{1}{n}\sum_{i\in[n]}\phi(x_i)\right]\otimes\left[\phi(x_{j+i})- \frac{1}{n}\sum_{i=\in[n]}\phi(x_{j+i})\right]},\;\; 0\leq j\leq n-1,   \end{equation}
and, hence, approximate $\HKoop$, the inverse problem \eqref{eq:inv_prob} is typically ill-posed and to learn well, we need to regularize. Here we focus on the Tikhonov and rank regularization, known to be statistically optimal for transfer operator learning in the form of \textit{Reduced Rank Regression} (RRR), \citep{kostic2023sharp}. More precisely, we minimize the regularized excess risk w.r.t. rank at most $r$ operators in $\RKHS$, i.e.
\begin{equation}\label{eq:reg_ex_risk}
\min_{{\rm rank}(\Estim)\leq r} \norm{\anfun(\generator)\TJ\TS\!-\!\TJ\TS\EEstim}^2_{\HS{\RKHS,\Liishort}} + \reg \norm{\Estim}_{\HS{\RKHS,\RKHS}}^2
\end{equation}
which, due to Eckart-Young-Mirsky theorem, c.f. \citep{Kostic2022}, can be expressed in a closed form
\begin{equation}\label{eq:population_rrr}
\RRR= (\Cxy{0}+\reg I)^{-1/2}\SVDr{(\Cxy{0}+\reg I)^{-1/2}\Yx},   
\end{equation}
where $\SVDr{\cdot}$ denoted the $r$-truncated singular value decomposition. Now, by replacing the population time-lagged cross-covariances with their empirical counterparts we obtain the RRR estimator of $\anfun(\generator)$
\begin{equation}\label{eq:empirical_rrr}
\ERRR= (\ECxy{0}+\reg I)^{-1/2}\SVDr{(\ECxy{0}+\reg I)^{-1/2}\EYx},\;\text{ where }\; \EYx = \textstyle{\weight_0\ECxy{0}+\sum_{j\in[\ell]}}[\,\weight_j\ECxy{j}+\weight_{-j}\ECxy{j}^*\,], 
\end{equation}
and $\ell<n-2$ is chosen small enough so that all empirical time-lagged cross-covariances with non-zero weighs are well-defined.

In the reminder of this section we show how to compute the estimator \eqref{eq:empirical_rrr} and its eigenvalue decomposition in two typical settings: (i) \textit{primal} when $m=\dim(\RKHS)\leq n$ and (ii) \textit{dual} when $n<m=\dim(\RKHS)\leq \infty$, which are summarized in Algorithms \ref{alg:primal} and \ref{alg:dual}, respectively. 

To that end, we introduce the sampling operator $\ES\colon\RKHS\!\to\!\R^{n}$ and its adjoint $\ES^*\colon\R^{n}\!\to\!\RKHS$
\begin{equation}\label{eq:sampling_op}
    \ES h\!=\!\tfrac{1}{\sqrt{n}} (h(x_{i\!-\!1}))_{i\in[n]}\text{ and } \ES^*v \!=\! \tfrac{1}{\sqrt{n}}\textstyle{\sum_{i\in[n]}}v_i\fH(x_{i\!-\!1}),
\end{equation}
which, after some basic computations, allow one to elegantly express the empirical time-lagged cross-covariances as $\ECxy{j}{=}\frac{n}{n-j}\ES^*\EJ(\sum_{i\in[n]} e_{i}e_{i+j}^\top)\EJ\ES$, where $(e_i)_{i\in[n]}{\subset}\R^n$ is the standard basis and $\EJ\in\R^{n\times n}$ is the orthogonal projector onto the complement of $e=\sum_{i} e_i$. Therefore, recalling \eqref{eq:emp_cross_cov}, we can define the Toeplitz matrix $\toeplitz_n\in\R^{n\times n}$ associated to scaled coefficients $\weight_j'=\tfrac{n}{n-j}\weight_j$, for $|j|\leq\ell$, that is 
\begin{equation}\label{eq:topelitz}
(\toeplitz_n)_{i,i+j} := \begin{cases}
(n\weight_j)/(n{-}|j|)  &, i\in[n], -\ell\leq j \leq \ell, \\
0 &, \text{ otherwise},
\end{cases}
\end{equation}
to express $\EYx = \ES^*\EJ \toeplitz_n\EJ\ES$ and prove following two theorems on eigenvalue decomposition of \eqref{eq:empirical_rrr} in both settings.

\begin{theorem}\label{thm:primal}
Let $\dt>0$, and let $\Data=(x_i)_{i=1}^n$ be a equally $\dt$-time spaced trajectory of a dynamical system generated by $\generator$ at a stationary distribution $\im$. Assuming that the RKHS $\RKHS$ is generated by the kernel $k(x,x')=z(x)^\top z(x')$, let $\Z = [z(x_1)\,\vert\,\cdots\,\vert\,z(x_n)]$ be the data matrix in the representation space. If by $\textsc{C}_\reg=\tfrac{1}{n}\Z\EJ\Z^\top+\reg I\in\R^{m\times m}$ and $\W = \tfrac{1}{n}\Z\EJ \toeplitz_n \EJ\Z^\top\in\R^{m\times m}$ we denote the regularized covariance and weighted time-lagged cross-covariance matrices, respectively, with $\toeplitz_n$ given in \eqref{eq:topelitz}, and if $\V_r=[v_1\,\vert\,\cdots\,\vert\,v_r]\in\R^{m\times r}$ consists of the eigenvectors corresponding to the largest eigenvalues $\sigma_i^2$ satisfying the generalized positive definite eigenvalue problem
\begin{equation}\label{eq:primal_gep}
\W \W^\top v_i = \sigma_i^2 \textsc{C}_\reg v_i,\;\;\text{ normalized s.t. } v_i^\top\textsc{C}_\reg v_i = 1,\;\; i\in[r],   
\end{equation}
then, assuming that the matrix $\V_r^\top \W \V_r \in\R^{r\times r}$ is non-defective, the spectral decomposition $(\widehat{\nu}_i,\elefun_i,\erefun_i)_{i\in[r]}$ of the RRR estimator \eqref{eq:empirical_rrr} is given by the eigenvalue decomposition $(\widehat{\nu}_i,\levec_i,\revec_i)_{i\in[r]}$ of the matrix $\V_r^\top \W \V_r \in\R^{r\times r}$, by $\elefun = z(\cdot)^\top\W\V_r\levec_i$ and $\erefun=z(\cdot)^\top \V_r\revec_i$. Moreover, if $\anfun(\generator)$ is self-adjoint, the eigenvalues $\nu_i$ are real, while if it is skew-adjoint, they are purely imaginary. 
\end{theorem}
\begin{proof}
Observing that the operators \eqref{eq:emp_cross_cov} on $\RKHS=\Span(z_i)_{i\in[m]}$ are isometrically isomorphic to a matrices $\textsc{C}_j$ computed by replacing $\fH$ by $z$, we conclude that estimator $\ERRR$ is isometrically isomorphic to $m\times m$ matrix $\textsc{C}_\reg^{-1/2}\SVDr{\textsc{C}_\reg^{-1/2}\W}$ in basis $(z_i)_{i\in[N]}$. But then, for $v_i$'s given by \eqref{eq:primal_gep} we have that $\textsc{C}_\reg^{1/2}v_i$ are the leading left singular vectors of $\textsc{C}_\reg^{1/2}\W$, and therefore 
\begin{equation*}
\textsc{C}_\reg^{-1/2}\SVDr{\textsc{C}_\reg^{-1/2}\W}=\textsc{C}_\reg^{-1/2}\big[\textsc{C}_\reg^{1/2}\V_r \V_r^\top \textsc{C}_\reg^{1/2}\big]\textsc{C}_\reg^{-1/2}\W=\V_r\V_r^\top\W.  
\end{equation*}
To conclude the proof it suffices to apply the characterization of the low-rank eigenvalue problems, \citep{SS1990}.
\end{proof}

\begin{algorithm}[h!]
\caption{Primal Toeplitz RRR} \label{alg:primal}
\begin{algorithmic}[1]
\REQUIRE 
dictionary of functions $(z_i)_{i\in[m]}$, Toeplitz coefficients $(\weight_i)_{i=-\ell}^{\ell}$ and hyperparameters $\reg>0$ and $r\in[n]$.
\STATE Compute $\Z = [z(x_1)\,\vert\,\ldots\,\vert z(x_{n})]\in\R^{m\times n}$
\STATE Remove the column-wise mean, i.e. $\Z\leftarrow \Z\EJ$
\STATE Apply Toeplitz matrix \eqref{eq:topelitz} to $\Z$, i.e. $\Z_\anfun\leftarrow \Z \toeplitz_n$ 
\STATE Compute $\W{=}\tfrac{1}{n}\Z_\anfun\Z^\top$ and $\textsc{C}_\reg{=}\tfrac{1}{n}\textsc{Z}\textsc{Z}^\top\!\!{+}\reg I$ 
\STATE Solve eigenvalue problem $\textsc{H}\textsc{H}^\top v_i{=}\sigma_i^2 \textsc{C}_\reg v_i$, $i\in[r]$ 
\STATE Normalize $v_i\leftarrow v_i / (v_i^\top \textsc{C}_\reg v_i)^{1/2}$, $i\in[r]$
\STATE Form $\V_r = [v_1\,\vert\,\ldots\,\vert\,v_r]\in\R^{N\times r}$
\STATE Compute eigentriplets $(\widehat{\nu}_i,w_i^l,w_i^r)$ of $\textsc{V}_r^\top \W \textsc{V}_r$
\STATE Construct $\elefun_i {=} z^\top \textsc{W}\textsc{V}_r w_i^l$ and $\erefun_i {=} z^\top \textsc{V}_r w_i^r$
\ENSURE Estimated eigentriples $(\widehat{\nu}_i,\elefun_i,\erefun_i)_{i\in[r]}$ of $\anfun(\generator)$ 
\end{algorithmic}
\end{algorithm}

Analyzing Algorithm \ref{alg:primal}, and recalling that $r\ll m\leq n$, we note that once the data is embedded using the dictionary representation, the main computational cost lies in lines 3-5. First, concerning line 3, since applying the Toeplitz matrix can be done efficiently, either via sparse matrix multiplication if $\ell\leq\log n$, resulting in $\bigO(m\,n\,\ell)$ floating point operations, or, otherwise, via fast Fourier transform (FFT), the worst-case computational complexity is of order $\bigO(m\,n\,\log n)$. On the other hand, computation of line 4 is $\bigO(m^2 n)$ is dominating the one of line 5 $\bigO(m^3)$. Therefore, the total complexity of the primal algorithm is $\bigO(m n (m\vee (\ell\wedge \log n))$.

Next, we consider the case $\dim(\RKHS)>n$, and derive dual Alg. \ref{alg:dual} applicable also to infinite-dimensional $\RKHS$. To do so, we need to perform computations in the "sample" space. That is, we rely on the reproducing property $h(x){=}\scalarpH{h,\fH(x)}$ and kernel Gram matrix $\Kx=[k(x_i,x_j)]_{i,j\in[n]} \!\in\! \R^{n\times n}$.

\begin{theorem}\label{thm:dual}
Let $\dt>0$, and let $\Data=(x_i)_{i=1}^n$ be an equally $\dt$-time spaced trajectory of a dynamical system generated by $\generator$ at a stationary distribution $\im$. Assuming that the RKHS $\RKHS$ is generated by the kernel $k(x,x')$, let $\Kx = [k(x_i,x_j)]_{i,j\in[n]}$ be the kernel Gram matrix, while $\overline{\Kx} = \tfrac{1}{n}\EJ\Kx\EJ$ and $\overline{\Kx}_\reg=\overline{\Kx}+\reg\,I\in\R^{n\times n}$ be its centered and regularized versions, while $\toeplitz_n$ is given in \eqref{eq:topelitz}. If $\U_r=[u_1\,\vert\,\cdots\,\vert\,u_r]\in\R^{m\times r}$ consists of the eigenvectors corresponding to the largest eigenvalues $\sigma_i^2$ satisfying the generalized eigenvalue problem
\begin{equation}\label{eq:dual_gep}
\toeplitz_n \overline{\Kx} \toeplitz_n^H \overline{\Kx} u_i = \sigma_i^2 \overline{\Kx}_\reg u_i,\;\;\text{ normalized s.t. } u_i^\top\overline{\Kx}\overline{\Kx}_\reg u_i = 1,\;\; i\in[r],   
\end{equation}
and $\V_r = \overline{\Kx} \U_r$ then, assuming that $\V_r^\top \toeplitz_n \V_r \in\R^{r\times r}$ is non-defective matrix, the spectral decomposition $(\widehat{\nu}_i,\elefun_i,\erefun_i)_{i\in[r]}$ of the RRR estimator \eqref{eq:empirical_rrr} is given by the eigenvalue decomposition $(\widehat{\nu}_i,\levec_i,\revec_i)_{i\in[r]}$ of the matrix $\V_r^\top \toeplitz_n \V_r \in\R^{r\times r}$, by $\elefun = \ES^*\toeplitz^H\V_r\levec_i$ and $\erefun=\ES^* \U_r\revec_i$. Moreover, if $\anfun(\generator)$ is self-adjoint, the eigenvalues $\nu_i$ are real, while if it is skew-adjoint, they are purely imaginary. 
\end{theorem}
\begin{proof}
Recalling \eqref{eq:empirical_rrr} and the definition of the sampling operators \eqref{eq:sampling_op}, start by writing 
$(\ECxy{0}+ \reg I_{\RKHS})^{-\frac{1}{2}}\EYx$ as  $(\ES ^*\EJ\ES + \reg I_{\RKHS})^{-\frac{1}{2}}\ES\EJ ^*\toeplitz_n\EJ\ES  = \ES ^*\EJ(\EJ\ES \ES ^*\EJ+ \reg I_{\RKHS})^{-\frac{1}{2}}\toeplitz_n \EJ\ES  =  \ES ^*\EJ\overline{\Kx}_\reg^{-\frac{1}{2}}\toeplitz_n\EJ\ES $. Its leading singular values $\sigma_1\geq\ldots\geq\sigma_r$ and the corresponding {\em left} singular vectors $g_1,\ldots,g_r\in\RKHS$ are obtained by solving the eigenvalue problem
\begin{equation}\label{eq:evp_1}
 \left(\ES^* \EJ\overline{\Kx}_\reg^{-\frac{1}{2}}\ET\EJ\ES\right)\left(\ES ^*\EJ\overline{\Kx}_\reg^{-\frac{1}{2}}\ET\EJ\ES\right)^{*} g_i  = \sigma_i^2 g_i,\; i\in[r].
\end{equation}
From the above equation, clearly $g_i\in\range(\ES ^*\EJ \overline{\Kx}_\reg^{-\frac{1}{2}}) = \range(\ES ^* \EJ\overline{\Kx}_\reg^{\frac{1}{2}})$, and we can represent the singular vectors as $g_i = \ES ^*\EJ \overline{\Kx}_\reg^{\frac{1}{2}} u_i$ for some $u_i\in\R^n$, $i\in[r]$. Therefore, substituting $g_i = \ES ^* \EJ \overline{\Kx}_\reg^{\frac{1}{2}} u_i$ in~\eqref{eq:evp_1} and simplifying  we obtain
\begin{equation}\label{eq:evp_2}
\ET\overline{\Kx}\ET^H\overline{\Kx}  u_i = \sigma_i^2 \overline{\Kx}_\reg u_i,\; i\in[r].
\end{equation}
Solving~\eqref{eq:evp_2} and using that $g_i = \ES ^*\EJ \overline{\Kx}_\reg^{\frac{1}{2}} u_i$, one obtains  $(\sigma_i^2,g_i)$, $i\in[r]$, the solutions of the eigenvalue problem \eqref{eq:evp_1}. In order to have properly normalized $g_i$, it must hold for all $i\in[r]$ that
\begin{equation}\label{eq:eigenvector_normalization}
1 =g_i^*g_i = u_i^\top \overline{\Kx}_\reg^{\frac{1}{2}} \EJ\ES  \ES ^*\EJ \overline{\Kx}_\reg^{\frac{1}{2}}  u_i = u_i^\top  \overline{\Kx} \overline{\Kx}_\reg u_i.
\end{equation}

Next, the subspace of the leading left singular vectors is $\range( \ES ^*\EJ \overline{\Kx}_\reg^{\frac{1}{2}} U_r )$, and, since the columns of $U_r$ are properly normalized, the orthogonal projector onto its range is given by $\Pi_r := \ES ^* \EJ\overline{\Kx}_\reg^{\frac{1}{2}} \U_r \U_r^\top \overline{\Kx}_\reg^{\frac{1}{2}} \EJ \ES $. We therefore have that $[\![ \ECx_\reg^{-\frac{1}{2}}\EYx ]\!]_r =  \Pi_r \ECx_\reg^{-\frac{1}{2}}\EYx = 
\ES ^* \EJ \overline{\Kx}_\reg^{\frac{1}{2}} \U_r \U_r^\top \overline{\Kx} \ET\EJ\ES $. Thus, defining $\V_r := \overline{\Kx} \U_r$, we conclude that 
\begin{equation*}
    \ERRR = \ECx_\reg^{-\frac{1}{2}}  \ES ^*\EJ \overline{\Kx}_\reg^{\frac{1}{2}} \U_r \V_r^\top \ET\EJ\ES  = \ES ^*\EJ \U_r \V_r^\top \ET\EJ\ES. 
\end{equation*}

So, to compute the spectral decomposition, we again apply the theorem on low-rank eigenvalue problems \citep{SS1990} and conclude that eigentriplets of $\ERRR$ are obtained by decomposing 
\begin{equation}\label{eq:small_evp}
\V_r^\top \ET\EJ\ES \ES ^*\EJ \U_r = \V_r^\top \ET\overline{\Kx}\U_r = \V_r^\top \ET\V_r = \textstyle{\sum_{i\in [r]}} \efeval_i\revec_i[\levec_i]^\top, \quad \text{ where }\quad [\levec_i]^\top \revec_i=1,\; i\in[r].
\end{equation}
Namely, left and right eigenfunction of $\ERRR$ corresponding to eigenvalue $\efeval_i$ are $\ES^*\EJ\ET^H\V_r\levec_i$ and $\ES^*\EJ\U_r\revec_i$, respectively. So, to conclude the proof it remains to scale the vectors, making them biorthogonal (paying attention to complex conjugation, since $\ET$ is complex and not necessarily  hermitian)
\[
[\levec_i]^H\V_r^\top \ET\EJ\ES\ES^*\EJ\U_r\revec_i = [\levec_i]^H\V_r^\top \ET\V_r\revec_i = \efeval_i,  
\]
where the last equality is due to \eqref{eq:small_evp}. Assuming, without loss of generality, that all $\efeval_i$'s are nonzero, and normalizing, completes the proof.
\end{proof}

\begin{algorithm}[h!]
\caption{Dual Toeplitz RRR} \label{alg:dual}
\begin{algorithmic}[1]
\REQUIRE
kernel $k$, Toeplitz coefficients $(\weight_i)_{i=-\ell}^{\ell}$, hyperparameters $\reg>0$ and $r\in[n]$. 
\STATE Compute kernel Gram matrix $\Kx\!=[k(x_i,x_j)]_{i,j\in[n]}\!\in\!\R^{n\times n}$
\STATE Center and normalize the Gram matrix, i.e. $\Kx\leftarrow \frac{1}{n}\EJ\Kx\EJ$
\STATE Apply Toeplitz matrix  $\Kx_\anfun\leftarrow \toeplitz_n \Kx \toeplitz_n^\top$ 
\STATE Solve eigenvalue problem $\textsc{K}_{\anfun} \Kx u_i{=}{\sigma}_i^2 \Kreg u_i$, $i\in[r]$, where $\textsc{K}_{\reg}=\Kx {+}\reg I$
\STATE Normalize $u_i\leftarrow u_i / (u_i^\top \Kx\Kreg u_i)^{1/2}$, $i\in[r]$
\STATE Form $\U_r = [u_1\,\vert\,\ldots\,\vert\,u_r]\in\R^{n\times r}$ and compute $\V_r {=} \Kx \U_r$
\STATE Compute eigentriples $(\widehat{\nu_i},w_i^l,w_i^r)$ of $\textsc{V}_r^\top \textsc{M} \V_r$
\STATE Construct $\elefun_i \!=\!\ES^*\toeplitz^\top\V_r \levec_i /\overline{\nu}_i$ and $\erefun_i \!=\!\ES^*\U_r \revec_i$
\ENSURE Estimated eigentriples  $(\widehat{\nu}_i,\elefun_i,\erefun_i)_{i\in[r]}$ of $\anfun(\generator)$ 
\end{algorithmic}
\end{algorithm}

Recalling that we can efficiently perform multiplication with Toeplitz matrices, the most expensive computation in the dual algorithm is in line 4. While naive computations results in cubic complexity w.r.t sample size $n$, using classical iterative solvers, like Lanczos or the generalized Davidson method to compute the leading eigenvalues of the generalized eigenvalue problem, when $r\ll n$ the cost can significantly be reduced, c.f.~\citep{HLAbook}. Furthermore, as proposed in \citep{Turri2024}, randomized algorithms can be used to solve problems of the form \eqref{eq:dual_gep}, leading to efficient numerical routines for the implementation of both Algorithms \ref{alg:primal} and \ref{alg:dual}. 

Once the spectral decomposition of $\anfun(\generator)$ is estimated, given an observable $h\in\RKHS$ we can efficiently approximate 
\begin{equation}\label{eq:spectral_filter}
[\anfun(\generator)]^s h \approx \textstyle{\sum_{i\in[r]}} \,\efeval_i^s\,\scalarpH{\elefun_i,h}\erefun_i\in\RKHS.
\end{equation}

Thus, we can think of our data-driven method as a general Toeplitz based approach to build general Krylov subspace methods, c.f. \citep{saad2011numerical}, for general class of (stochastic) dynamical systems. We believe that this approach can lead to data-driven generalizations of diverse spectral filter methods for fast numerical computations, such as Chebyshev \citep{di2016efficient,zhou2007chebyshev} and resolvent based \citep{polizzi2009density} FEAST filters. This exciting new direction of research could offer new classes of efficient methods that overcome the curse of dimensionality, a typical bottleneck in practical dynamical systems.

Finally, in light of \eqref{eq:solutions} for an unknown SDE in \eqref{Eq: SDE}, whenever $\anfun$ is bijective on the spectrum of $\generator$, by solving $\eeval_i=\anfun(\widehat{\nu}_i)$, $i\in[r]$, Algorithms \ref{alg:primal} and \ref{alg:dual} enable the construction of approximate solutions from a single (long) simulated trajectory by estimating the spectrum of $L$. In such cases, we can approximate dynamics by
%by estimating the dominant spectrum of $\generator$ from a single (long) trajectory of the system obtained through simulations, Algorithms \ref{alg:primal} and \ref{alg:dual} allow constructing solutions as
\begin{equation}
\label{eq:prediction_mean}
\EE[h(X_t)\,\vert\,X_0\!=\!x] \approx \textstyle{\sum_{i\in[r]}} \,e^{\eeval_i t}\,\scalarpH{\elefun_i,h}\erefun_i(x),
\end{equation}
where $\langle \elefun_{i}, h \rangle_{\mathcal H}$ can be computed on the training set via the kernel trick, see \cite{Kostic2022}. As discussed above, this approach is particularly interesting for high-dimensional state spaces, where classical numerical methods become unfeasible due to the curse of dimensionality, making data-driven methods a key tool in fields like molecular dynamics, \cite{schutte2023overcoming}. 

% We prove in Sec. \ref{sec:bounds} that the precision of our method does not depend on the state dimension, but on intrinsic effective dimension of the process, which, together with its linear complexity w.r.t. $d$, makes it an attractive approach in such problems. Finally, we remark that \eqref{eq:prediction_mean} enables forecasting of full state distributions and not just the mean (e.g. $f$ can be an indicator function), noting that this formula extends to all $\mathcal{L}^2_\pi$ functions, at the price of an additional projection error, which leads to predicting the evolution of distributions, c.f. \citep{Kostic2024forecasting}.

\section{Statistically consistent spectral Toeplitz estimators}
\label{sec:consistency}
In this section we first present a general result on the statistical consistency of Toeplitz based estimators, paving the way towards their learning theory, and then discuss specific choices of Toeplitz symbols $\toepsymb$, leading to the generator transforms $\anfun$ summarized in \cref{tab:summary}. 

% \begin{definition}[$\beta$-mixing]
% Let $(X,\mathcal B,\mu,\{\Phi^t\}_{t\in\mathbb R})$ be a deterministic dynamical system,
% where $\Phi^t : X \to X$ is a measurable flow and $\mu$ is a $\Phi^t$-invariant probability
% measure. Define the induced stochastic process by
% \[
% X_t(x) := \Phi^t(x), \qquad x \sim \mu .
% \]
% For $\tau \ge 0$, the $\beta$-mixing (absolute regularity) coefficient is defined as
% \[
% \beta(\tau)
% \;:=\;
% \sup_{t \in \mathbb R}
% \left\|
% \mathcal L_\mu\!\left(
% (X_s)_{s \le t},
% (X_s)_{s \ge t+\tau}
% \right)
% -
% \mathcal L_\mu\!\left((X_s)_{s \le t}\right)
% \otimes
% \mathcal L_\mu\!\left((X_s)_{s \ge t+\tau}\right)
% \right\|_{\mathrm{TV}},
% \]
% where $\mathcal L_\mu$ denotes the law induced by the random initial condition
% $x \sim \mu$.  
% The system is said to be \emph{$\beta$-mixing} (or \emph{absolutely regular}) if
% $
% \beta(\tau) \xrightarrow[\tau\to\infty]{} 0.
% $
% \end{definition}
\begin{definition}[$\beta$-mixing]  %(absolute regularity)]
Let $(\Omega,\mathcal F,\mathbb P)$ be a probability space and let
$\{X_t\}_{t\in\mathbb R}$ be a stochastic process with values in a measurable space
$(\mathcal{X},\mathcal B)$. For $\tau \ge 0$, the $\beta$-mixing coefficient is defined by
\[
\beta_{X}(\tau)
\;:=\;
\sup_{t \in \mathbb R}
\left\|
\mathcal L\!\left(
(X_s)_{s \le t},
(X_s)_{s \ge t+\tau}
\right)
-
\mathcal L\!\left((X_s)_{s \le t}\right)
\otimes
\mathcal L\!\left((X_s)_{s \ge t+\tau}\right)
\right\|_{\mathrm{TV}},
\]
where $\mathcal L$ denotes the joint law induced by $\mathbb P$.
The process is said to be \emph{$\beta$-mixing} (or \emph{absolutely regular}) if
$
\beta_{X}(\tau) \xrightarrow[\tau\to\infty]{} 0.
$
\end{definition}

A deterministic dynamical system $(\mathcal{X},\mathcal B,\mu,\{\Phi^t\}_{t\in\mathbb R})$ fits this
framework by setting $X_t(x)=\Phi^t(x)$ with randomness induced solely by the initial
condition $x\sim\mu$.

Since in applications, we often observe a dynamical system at an evenly spaced sampling rate $\Delta t>0$, we define
$\beta_{X_{\cdot\Delta t}}(\bar{\tau})$ as the $\beta$-mixing coefficient of the discrete-time
process $(X_{n\Delta t})_{n\in \mathbb{N}}$, with lag $\tau = \bar{\tau}\Delta t$, $\bar{\tau}\in\mathbb{N}$.

\begin{theorem}[Consistency of $\ERRR$]\label{thm:consistency}
Let $\dt>0$, and let $\Data=(x_i)_{i=1}^n$ be a equally $\dt$-time spaced trajectory of a dynamical system generated by $\generator$ at a stationary distribution $\im$. Let $\RKHS$ be an RKHS generated by the kernel $k(x,x')$ which is either universal or of the form $k(x,x')=z(x)^\top z(x')$, for some $(z_k)_{k\in[m]}$ such that $(z_k)_{k\in\N}$ forms a basis of $\Lii$. Assume that
\begin{enumerate}
    \item[i)] the dynamical system $(X_t)_{t\in\R_0}$ is beta mixing, and the sequences $l(n),s(n),\bar{\tau}(n) \to \infty$ as $n\to \infty$ such that $n-l > 2 s \bar{\tau}$ and $(s-1) \beta_{X_{\cdot\dt}}(\bar{\tau} -1 ) \to 0 $ as $n\to \infty$.  
    \item[ii)] $\toepsymb$ is analytic in the neighborhood of some spectral set (typically spectrum or numerical range) and $\toepsymb_\ell$ converges to it as $\ell\to\infty$.
\end{enumerate}
The the operator norm error $\|\anfun(\generator)\TS-\TS\ERRR\|_{\HS{\RKHS,\Liioshort}}^2$ converges to zero in probability as $l,s,n\to\infty$. Consequently, for every $i\in[r]$ estimator's eigenvalue $\efeval_i$ converges to a point in the spectrum of $\anfun(\generator)$, and, if the target point is a simple eigenvalue, then also $\erefun_i$ convergence to the corresponding eigenfunction of $\anfun(\generator)$.  
\end{theorem}

\begin{proof}
For any estimator $\EEstim$ of $\anfun(\generator)$ we have
\begin{align*}
\norm{\anfun(\generator)\TJ\TS\!-\!\TJ\TS\EEstim}_{\RKHS \to \Liioshort}  &\leq \norm{(\anfun(\generator) - \anfun_\ell(\generator))\TJ\TS}_{\RKHS \to \Liioshort}  + \norm{\anfun_\ell(\generator)\TJ\TS\!-\!\TJ\TS\EEstim}_{\RKHS \to \Liioshort}\\
&\leq \norm{(\anfun(\generator) - \anfun_\ell(\generator))}_{\Liioshort \to \Liishort}\norm{\TJ\TS}_{\RKHS \to \Liioshort}  + \norm{\anfun_\ell(\generator)\TJ\TS\!-\!\TJ\TS\EEstim}_{\RKHS \to \Liioshort}\\
&\leq {\rm err}_\ell(\anfun) \norm{\TJ}_{\Liioshort \to \Liioshort}  \norm{\TS}_{\RKHS \to \Liioshort}  + \norm{\anfun_\ell(\generator)\TJ\TS\!-\!\TJ\TS\EEstim}_{\RKHS \to \Liioshort}.
\end{align*}
Note that $\norm{\TJ}_{\Liioshort \to \Liioshort} = 1$ and $\norm{\TS}_{\RKHS \to \Liioshort}\lesssim 1$. In addition under (ii) we have ${\rm err}_\ell(\anfun) \to 0$ as $l \to \infty$.

We consider now the term $\norm{\anfun_\ell(\generator)\TJ\TS\!-\!\TJ\TS\EEstim}_{\RKHS \to \Liioshort}$ when $\EEstim$ is the RRR estimator defined in \eqref{eq:empirical_rrr}. Under (i) the dynamical system is $\beta$-mixing. Hence for any $\delta\in (0,1)$ and $s\geq 1$, there exists a large enough $\bar{\tau}>0$ such that $\delta > 2(s-1) \beta_{X_{\cdot \dt}}(\bar{\tau}-1)$. Next as we consider the asymptotic $n\to \infty$, we can assume that $n$ is large enough such that $n-l \geq 2 s \bar{\tau}$. Then applying Proposition C.7 in \citep{kostic2025laplace} guarantees the convergence in probability: $\norm{\anfun_\ell(\generator)\TJ\TS\!-\!\TJ\TS\EEstim}_{\RKHS \to \Liioshort} \stackrel{\mathbb{P}}{\to} 0$ as $n\to \infty$. 
\end{proof}

Next, we discuss several useful examples of Toeplitz estimators. While some were essentially known (though typically without learning in $\Liio$), many are new, stemming from our new framework of Toeplitz linear algebra for generators of continuous Markov semigroups.

\textbf{Transfer operators.} As a first example we revisit classical Koopman/Transfer operators, which we can split in self-adjoint and skew-adjoint part
\begin{equation*}
\LKoop_t =e^{t \generator} = \frac{e^{t \generator}+e^{t \generator^*}}{2} + \frac{e^{t \generator}-e^{t \generator^*}}{2}, 
\end{equation*}
which is reminiscent of the highly influential paper of \citet{bai2003hermitian} on Hermitian-Skew Hermitian splittings for solving linear problems. 

Now, clearly, for $\generator^*=-\generator$ we have that $A_t$ is unitary operator with hyperbolic-trigonometric splitting into self-adjoint part $\cosh(t \generator)$ and skew-adjoint part $\sinh(t\generator)$. On the other hand, when $\generator^*=\generator$ we have that $A_t =\cosh(t \generator)$. As a consequence, we obtain the first four rows of \cref{tab:summary}. Furthermore, for each of these estimators we have $\toepsymb=\toepsymb_1 =\toepsymb_\ell$ for all $\ell\in\N$ directly obtaining the consistency of the corresponding empirical estimators, which, importantly, as shown in Theorems \ref{thm:primal} and \ref{thm:dual} preserve location of the spectrum of $\generator$ on either real or imaginary axis. 

\textbf{Transfer operator's resolvent.} Eigenvalues, while informative about long-term behavior, fail to capture transient dynamics of the full time evolution of the process whenever the transfer operator is non-normal, that is when $A_t A_t^* \,{\neq}\, A_t^*A_t$, \citep{TrefethenEmbree2020}. In contrast, the resolvent of $A_t$ defined by $\omega\mapsto(\omega\Id\!-\!A_t)^{-1}$, $\omega\in\rho(A_t)$, provides a more comprehensive view of the dynamics, making it the core object of spectral theory. In particular, for non-normal transfer operators the transient growth of powers
can be bounded, c.f. \citep{el2002extremal},  
\begin{equation}\label{eq:kreiss}
{\cal K}(A_t)\leq \sup_{k\in\N_0}\|A_t^k\|_{\Liioshort\to\Liioshort}\leq \frac{e}{2}[{\cal K}(A_t)]^2\quad\text{ where }\quad{\cal K}(A_t):=\sup_{\Re(\mu)>0} \|(e^{\mu}-A_t)^{-1}\|(e^{\Re(\mu)}-1)   
\end{equation}
denotes the Kreiss constant. This constant is completely determined by the resolvent growth outside the unit disc and is fundamental tool in understanding how the system returns to equilibrium after perturbations. 

To show estimate the Kreiss constant, we consider the resolvent of a Koopman/Transfer operator $\anfun(\generator)= (e^{\mu}-A_t)^{-1} = (e^{\mu}-e^{t L})^{-1}$ for $\Re(\mu)>0$. Since for $z\in\mathbb{D}$ we have $|ze^{-\mu}|<1$, we can expand a Toeplitz symbol $\toepsymb(z)=(e^{\mu}-z)^{-1}=e^{-\mu}(1-ze^{-\mu})^{-1}$ as classical Von Neumann series 
\begin{equation}\label{eq:toep_res_to}
\textstyle{\toepsymb(z)= \sum_{j=0}^{\infty} e^{-\mu} [ze^{-\mu}]^j=\lim_{\ell\to\infty}=\sum_{j=0}^{\ell} e^{-(j+1)\mu} z^j= \lim_{\ell\to\infty}\toepsymb_\ell(z)},
\end{equation}
directly obtaining the consistency. While for a generic $\mu$, $\anfun$ is neither conjugate invariant nor equivariant, for $\mu>0$ and $\generator^*=\generator$ we can use symmetric Toeplitz matrices, obtained by replacing $z^j$ with $(z^j + \overline{z}^j)/2$ in \eqref{eq:toep_res_to}, yielding estimators with real spectrum. Finally, we note that a special case of this estimator ($r=n$ and learning in $\Lii$ with uncentered features) was introduced in \citep{hashimoto2020krylov} for discrete-time dynamical systems. 

\textbf{Generator's resolvent.} While powerful tools, transfer operator estimators are limited by the size of time-lag $\dt$, which obscures the time-scales smaller than $\dt$ and oscillatory frequencies higher than $1/(2\dt)$. Moreover, statistical learning guarantees for estimation of the transfer operators collapses when $\dt\to0$, see \citep{kostic2025laplace}, and one needs to resort to generator estimation to mitigate this issue. Since the differential operator $\generator$ is typically unbounded, it is more natural to work with the resolvent $(\shift\Id\!-\!\generator)^{-1}$, $\shift\in\rho(\generator)$. Our Toeplitz based framework naturally encompasses resolvent operator through its characterization via the Laplace transform, (see for instance \citep{Bakry2014}, equation (A.1.3)) as 
\begin{equation}\label{eq:resolvent}
\anfun(L){:=}(\shift\Id\!-\!\generator)^{-1} {=} \textstyle{\int_{0}^{\infty}} \LKoop_{t} e^{-\shift t}dt\;\implies\; \toepsymb_\ell(A_{\Delta_t}):= \sum_{j=0}^{\ell} \weight_j A_{\dt}^j\to \anfun(\generator)\text{ strongly when }\dt{\to} 0, \ell{\to}\infty,
\end{equation}
where \(\weight = (\weight_j)_{j=0}^{\ell}\) are weights given by the trapezoid rule with $\ell\geq1$ points and time-discretization %$(t_j)_{j=0}^{\ell}$ s.t. 
$\dt>0$, that is
\begin{equation}\label{eq:trapezoid}
t_j \!=\! j\dt\quad\text{ and }\quad
\weight_j \!=\!
\begin{cases}
    \frac{\dt}{2} \,e^{-\shift\,t_j} &\text{if $j\!\in\!\{0,\ell\}$,}\\
    \dt \,e^{-\shift\,t_j} & \text{if $1\leq j\!\leq\! \ell\!-\!1$}.
\end{cases}
%\omega_j = \begin{cases}
%    \frac{1}{2} &\text{if $j\in\{0,\ell\}$,}\\
%    1 & \text{if $1\leq j\leq \ell-1$.}
%\end{cases}
\end{equation}
Furthermore, if the generator is self-adjoint, for $\mu\in\R_+$ we can obviously use symmetric version of Toeplitz symbol 
\begin{equation}\label{eq:toep_res_sym}
\toepsymb_\ell(z) = \textstyle{\tfrac{1}{2}\big[\weight_0+\sum_{j=-\ell}^{\ell} \weight_{|j|} z^j\big],}    
\end{equation} 
ensuring that the estimated eigenvalues are real by construction. While the previous strong convergence result provides consistency in estimating $\|(\mu-\generator)^{-1}f\|_{\Liishort}$ for any $f\in\Lii$ and $\Re{\mu}>0$, stronger operator norm convergence of Theorem \ref{thm:consistency} can be shown for sectorial operators, c.f. see \eqref{eq:sectorial} and \eqref{eq:app_error_fun}, as discussed in \citep{kostic2025laplace} for learning in $\Lii$.

\textbf{Band-limited pseudo-inverse.} 
Finally, we consider the case of deterministic systems and the filtered pseudo-inverse
$\anfun(L)=P_{(\theta_{\min},\theta_{\max})}L_0^{-1}$, where $L_0^{-1}$ denotes the inverse
of the generator restricted to the orthogonal complement of its kernel and
$P_{(\omega_{\min},\omega_{\max})}$ is a spectral projector onto frequencies in the range
$(\omega_{\min},\omega_{\max})$. The corresponding Toeplitz symbol
$\toepsymb(z)=\one_{\{|{\rm Arg}(z)|\in[\theta_{\min},\theta_{\max}]\}}/{\rm Ln}\,z$
acts as a spectral band-pass filter, isolating oscillatory components while suppressing
low-frequency modes. As a consequence, the dominant spectrum of the resulting Toeplitz
estimator consists precisely of the frequencies in the prescribed range, making this
construction well suited for deterministic dynamics on simple attractors with a discrete
imaginary spectrum. While frequency-selective constructions for Koopman generators have been studied in \citep{Mezic2005,giannakis2019data}, the Toeplitz-symbol formulation and the corresponding estimators employed here are new and specific to our framework.

First, we note that for deterministic systems $\Spec(\generator){\in} i\R$, so we consider the Toeplitz symbol on the unit circle, using $\omega {\in }(-\pi,\pi]$ (principal branch of the complex logarithm), $0 {\leq} \omega_{\min} {<} \omega_{\max} {\leq} \pi$, and $\theta {= }\frac{\omega}{2\pi}$ being the frequency in Hz, that 
\begin{equation}
    \label{eq:toeplitzsymbol}
T(e^{i\omega}) = -i \cdot \frac{\mathbf{1}_{\omega_{\min} \leq |\omega| \leq \omega_{\max}}}{\omega}.
\end{equation}

For the $\ell$-truncated symbol $T_\ell$ we need the Fourier coefficients $a_j$, defined by
\[
a_j = \frac{1}{2\pi} \int_{-\pi}^{\pi} T(e^{i\omega}) e^{-i\,j\,E\omega} \, d\omega
= -\frac{i}{2\pi} \int_{-\pi}^{\pi} \frac{\mathbf{1}_{\omega_{\min} \leq |\omega| \leq \omega_{\max}}}{\omega} e^{-ij\omega} \, d\omega.
\]

By symmetry of the indicator domain and oddness of $1/\omega$,
\begin{align*}
a_j &= -\frac{i}{2\pi} \left[ \int_{\omega_{\min}}^{\omega_{\max}} \frac{e^{-ij\omega}}{\omega} \, d\omega 
+ \int_{-\omega_{\max}}^{-\omega_{\min}} \frac{e^{-ij\omega}}{\omega} \, d\omega \right] = -\frac{i}{2\pi} \int_{\omega_{\min}}^{\omega_{\max}} \frac{e^{-ij\omega} - e^{ij\omega}}{\omega} \, d\omega = -\frac{1}{\pi} \int_{\omega_{\min}}^{\omega_{\max}} \frac{\sin(j\omega)}{\omega} \, d\omega.
\end{align*}
Using the sine integral $\mathrm{Si}(x) = \int_0^x \frac{\sin t}{t} \, dt$,
\begin{equation}\label{eq:toep_inv}
a_j = -\tfrac{1}{\pi} \left[ \mathrm{Si}(j\omega_{\max}) - \mathrm{Si}(j\omega_{\min}) \right],\quad\text{ for }
\quad j \neq 0,\quad\text{ and }\quad a_0 = 0.
    \end{equation}
Since $\mathrm{Si}(-x) = -\mathrm{Si}(x)$, we have $a_{-j} = -a_j$, reflecting the odd symmetry of $T(e^{i\omega})$ in $\omega$.

For the truncated Toeplitz operator $T_\ell$ with symbol $T(e^{i\omega})$, the $\ell \times \ell$ truncation corresponds to taking the partial Fourier sum. Hence, when $T(e^{i\omega})$ has jump discontinuities at $\omega = \pm\omega_{\min}, \pm\omega_{\max}$, the truncation exhibits the Gibbs phenomenon: near each discontinuity, overshoots of approximately $8.95\%$ of the jump height persist as $\ell \to \infty$ \citep{hewitt1979gibbs}. To mitigate Gibbs oscillations in the truncated Toeplitz approximation, we can apply Jackson smoothing, by replacing $T_\ell(e^{i\omega})$ with a damped sum 
$T'_\ell(z) = \sum_{j=-\ell}^{\ell} a'_j z^{j}$,
where the Jackson damping factors  are chosen to decay smoothly with $|j|$ and ensure uniform convergence, that is we set
\begin{equation}\label{eq:jackson}
s_j=1 - \Big(\tfrac{|j|}{\ell+1}\Big)^2,\quad\text{ and scale}\quad a_j'= s_j\,a_j.    
\end{equation}

Indeed, since $T$ is a piecewise-$C^1$ with modulus of continuity bounded by $\delta/\omega_{\min}^2$ truncated symbol with Jackson factors satisfies, c.f.  \cite{zygmund2002trigonometric},
\[
\|T - T'_\ell\|_\infty \leq \frac{c}{\ell\omega_{\min}^2},
\]
where $C$ is a universal constant. So, we have eliminated the Gibbs overshoot at the cost of reducing the convergence rate from spectral to algebraic.

With this in mind, to apply Theorem \cref{thm:consistency}, since $T'_\ell$ is analytic on $\mathbb{T}$, it remains to assure that Toeplitz symbol $T$ is analytic on some spectral set. Since $L^*=L$ is normal, we can consider the spectrum of $L$, c.f. \eqref{eq:app_error_fun}. Clearly, the only problematic points for $\toepsymb$ are discontinuities $0\leq \omega_{\min}$ and $\omega_{\max}\leq\pi$. Therefore, whenever $i\omega_{\min},i\omega_{\max}\not\in\Spec(\generator)$, the error is analytic in the neighborhood of the spectrum and Theorem \ref{thm:consistency} applies to any deterministic beta-mixing dynamical system, necessarily on a simple attractor since spectrum of chaotic systems fills the whole imaginary axis.    

\textbf{Chebyshev filters for deterministic systems.}
We now focus on the case of possibly chaotic deterministic dynamical systems at equilibrium. By the spectral theorem for skew-adjoint operators  \citep{Kato} there exists a projection-valued measure \(E(\omega)\) such that
\[
L = \int_{\mathbb{R}} i\omega \, dE(\omega),
\qquad
A_t = \int_{\mathbb{R}} e^{i\omega t}\, dE(\omega),\quad\text{ and }\quad \anfun(L) = \int_{\mathbb{R}} f(i\omega)\, dE(\omega)
\]
holds for any bounded Borel function \(f\) on \(i\mathbb{R}\).

In contrast to the self-adjoint case, polynomial approximations in \(L\) are no longer appropriate; instead, trigonometric and Chebyshev-type expansions in the unitary group \(\{A_t\}_{t\in\mathbb{R}}\) provide the natural approximation framework. In particular, for chaotic deterministic systems, the spectral measure of \(L\) is typically purely continuous and may fill the entire imaginary axis. In such cases individual eigenvalues are absent or physically irrelevant, and classical eigenvalue-based methods such as EDMD fail to converge. Toeplitz-based trigonometric and Chebyshev filters, however, remain well-defined at the level of spectral measures. They enable stable approximation of spectral densities, band-limited projectors, and resolvent-type operators even when \(\sigma(L)=i\mathbb{R}\). This makes the approach particularly well suited for chaotic dynamics, where the goal is not to isolate eigenvalues but to extract coherent frequency bands and dynamically meaningful observables.

We consider a general (not necessarily even) Borel function $f : i\mathbb{R} \to \mathbb{C}$,
defined on the spectrum of $L$ and admitting a trigonometric expansion on the Nyquist interval
$[-\pi/\Delta t,\pi/\Delta t]$. Writing $\lambda = i\omega$, we decompose
\[
\anfun(i\omega) = \anfun_{\mathrm{e}}(i\omega) + \anfun_{\mathrm{o}}(i\omega),\quad\text{ 
where }\quad
\anfun_{\mathrm{e}}(i\omega) := \tfrac12\big(\anfun(i\omega) + \anfun(-i\omega)\big)
\quad\text{ and }\quad
\anfun_{\mathrm{o}}(i\omega) := \tfrac12\big(\anfun(i\omega) - \anfun(-i\omega)\big)
\]
are the even and odd parts, respectively. 

This leads us to \textbf{trigonometric expansions}. Namely, when there exist coefficients $\alpha_k,\beta_k \in \mathbb{C}$ such that
\[
\anfun(i\omega)
=
\alpha_0
+
\textstyle{\sum_{k=1}^\ell}
\left(
\alpha_k \cos(k\omega\Delta t)
+
\beta_k \sin(k\omega\Delta t)
\right),
\]
by the spectral theorem, and using
\[
\cos(k\Delta t\,L) = \tfrac12\big(A_{\dt}^k + A_{\dt}^{*k}\big),
\qquad
\sin(k\Delta t\,L) = \tfrac{1}{2\mathrm{i}}\big(A_{\dt}^k - A_{\dt}^{*k}\big),
\]
we obtain a Laurent--Toeplitz representation
\begin{equation}\label{eq:toep_trig}
\anfun(L) = \textstyle{\sum_{k=-\ell}^\ell a_k A_{\dt}^k},
\end{equation}
with coefficients
\[
a_0 = \alpha_0,
\qquad
a_k = \tfrac12\alpha_k + \tfrac{1}{2\mathrm{i}}\beta_k,
\qquad
a_{-k} = \tfrac12\alpha_k - \tfrac{1}{2\mathrm{i}}\beta_k,
\quad k\ge 1.
\]

Furthermore, the even and odd parts also admit independent \textbf{Chebyshev expansions}
\begin{equation}
\anfun_{\mathrm{e}}(i\omega)
=
\sum_{k=0}^\ell b_k \,\mathcal{T}_k(\cos(\omega\Delta t)),\qquad\text{ and }\qquad
\anfun_{\mathrm{o}}(i\omega)
=
\sin(\omega\Delta t)\sum_{k=0}^{\ell-1} c_k \,\mathcal{U}_k(\cos(\omega\Delta t)),
\end{equation}
where $\mathcal{T}_k$ and $\mathcal{U}_k$ denote the Chebyshev polynomials of the first and second
kind, respectively. Again using the spectral theorem we obtain
\begin{equation}\label{eq:toep_cheb}
\anfun(L)
=
\sum_{k=0}^\ell b_k\,
\mathcal{T}_k\!\left(\tfrac12(A_{\dt}+A_{\dt}^*)\right)
+
\sin(\Delta t\,L)
\sum_{m=0}^{M-1} c_m\,
\mathcal{U}_m\!\left(\tfrac12(A_{\dt}+A_{\dt}^*)\right).
\end{equation}

Since both families satisfy the three-term recurrences
\[
\mathcal{T}_{k+1}(B) = 2B\mathcal{T}_k(B) - \mathcal{T}_{k-1}(B),
\qquad
\mathcal{U}_{k+1}(B) = 2B\mathcal{U}_k(B) - \mathcal{U}_{k-1}(B),
\]
with $B = \tfrac12(A_{\dt}+A_{\dt}^*)$. Each multiplication by $B$ increases the Toeplitz bandwidth by one, so that $\anfun(L)$ is represented by a banded Toeplitz operator whose coefficients are generated by
local three-term recurrences.

Even when $\sigma(L)=i\mathbb{R}$, as is typical for chaotic deterministic dynamics, the operator
$\anfun(L)$ remains well defined as a bounded spectral multiplier. The Toeplitz representation provides
a data-driven realization of this functional calculus using only powers of $A_{\dt}$ and $A_{\dt}^*$, enabling
frequency-selective filtering and spectral isolation in the absence of point spectrum.

\paragraph{Implications.}
This formulation transforms functional calculus for skew-adjoint generators into structured
Toeplitz linear algebra. We conclude with few important remarks  on:
\begin{itemize}
    \item \textbf{Deterministic systems:} The skew-adjoint setting highlights the flexibility of Toeplitz-based Koopman estimators: by exploiting unitary structure and time-reversal equivariance, 
    we can approximate rich classes of operator-valued functions using only forward and backward time shifts. This accommodates arbitrary (even or odd) spectral filters, admits stable Chebyshev and Jackson--Fejér regularizations, and applies uniformly to quasiperiodic, mixing, and fully chaotic deterministic systems. In such a way we pave a way towards numerically stable, data-driven spectral methods that remain valid in the presence of continuous spectra and chaos, substantially extending the scope of Koopman analysis beyond classical finite-dimensional approximations. 
    \item \textbf{Jackson smoothing}: (1) Smearing of spectral discontinuities: The transition region near a jump in symbol has width $\mathcal{O}(1/\ell)$ in frequency, (2) Reduced convergence rate: Away from discontinuities, where the original Toeplitz truncation converges exponentially, the smoothed version converges only as $\mathcal{O}(1/\ell)$, (3) Practical implication: For large $\ell$, the smoothed $T_\ell$ better approximates the infinite Toeplitz operator's spectral properties near discontinuities, but may require larger $\ell$ for the same accuracy away from jumps, (4) Alternative smoothing kernels (Fej\'{e}r, Lanczos) offer different compromises between overshoot suppression and convergence rate \citep{lund1992sinc}, which will be the topic for future studies along other types of Toeplitz spectral filters of the generator.
    \item \textbf{Parallel eigensolvers:} Several applications require approximating family  of filters $(\anfun_\mu)_{\mu\in \C}$, for example when estimating the Kreiss constant via \eqref{eq:kreiss}. This task can become prohibitively expensive if each filter is treated independently. However, Theorems \ref{thm:primal} and \ref{thm:dual} show that while the matrix pencils of generalized positive definite eigenvalue problems in \eqref{eq:primal_gep} and \eqref{eq:dual_gep} depends nonlinearly on the parameter $\mu$, the leading matrix remains fixed and positive definite for all $\mu$. This structural invariance enables significant computational savings: factorizations or preconditioners for the leading matrix can be computed once and reused across all parameter values, and the computations for different $\mu$ can be parallelized efficiently. Techniques such as preconditioned iterative eigensolvers \citep{saad2011numerical} with recycling \citep{soodhalter2020survey}, and solvers able to exploit the nonlinear structure in $\mu$ \citep{sorensen1992implicit} can be further used to accelerate the solution process.
\end{itemize}

\section{Experiments}
\label{sec:exps}
%In this section, we illustrate the behaviour of our method. In particular, we implement dual Algorithm \ref{alg:dual} with universal approximation properties, and show that LaRRR was able to recover eigenvalues and eigenfunctions of the infinitesimal generator of the process. In the first experiment in 1D, we measure the error to the ground truth, while in the second (realistic molecular dynamics) one we recover consistent results with independent studies~\citep{deepTICA}.  

In this section, we demonstrate the usefulness of the Toeplitz based spectral estimators in contrast to standard EDMD and RRR estimators of the Koopman/Transfer Operators. Since the special case of generator's resolvent of stochastic systems in Examples \ref{ex: Langevin} and \ref{ex:OU} was already empirically studied with primal and dual algorithms in \citep{kostic2025laplace}, in what follows we focus on different transforms and a deterministic system of Example \ref{ex: Duffing} and study the spectrum of the Duffing oscillator on a simple attractor and on a strange attractor (chaotic regime). 

We recall that the generator is skew-adjoint, and hence the spectrum $\Spec(\generator)$ is purely imaginary. In both cases, we simulate the system at $\dt = 0.1\,[sec]$ and collect the points $\tilde{X}_i=(x_{t_0+i \dt},y_{t_0+i\dt})$ on the trajectory after some initial burn-in period $t_0>0$. Then we construct the training set $\Data=(X_i)_{i\in[n]}$ by taking $X_i = \tilde{X}_i+\sigma\,\xi_i$, with $\xi_i$ iid standard Gaussian and $\sigma>0$ being observational noise. 

In the following, we will use the primal algorithm, choosing features $m=100$ dimensional features made from up to $4$th order monomials of $x$ and $y$ coordinates in the past time window of length $10$. Without rank reduction, this estimator of Koopman operator is known as Extended Henkel DMD  \citep{colbrook2023multiverse}.  

\textbf{Simple attractor.} We consider parameters $\alpha=0.5$, $\beta=0.625$, $\gamma=2$, $\delta=1.5$ and $\omega=1$, yielding a stationary regime shown in \cref{fig:Duffing_1} (left). It is well know that the eigenvalues of the generator are $\lambda_k=i k \omega = i k$, $k\in\Z$, so the base frequency is $1/(2\pi)\approx 0.1592$. Since the standard Henkel and RRR estimators work well with perfectly observed $n=8000$ samples (noiseless samples), we consider more realistic setting of noisy observations and set $\sigma = 0.3$. The training trajectory $(\tilde{X}_i)_{i\in[n]}$ and noisy training samples $({X}_i)_{i\in[n]}$ are shown in \cref{fig:Duffing_1} (left). Once the models are trained we use as the initial point $(x_{t_0+(n+1)\dt},y_{t_0+(n+1)\dt})$ and test the prediction using \eqref{eq:prediction_mean} for the time $t=500 \dt = 50\,[sec]$. Results are shown for Koopman operator (baseline), Hyperbolic Sine (imaginary spectrum by design) and the Inverse projected on the Band (imaginary spectrum with focus on the lower frequency band) given in \eqref{eq:toep_inv}. \cref{fig:Duffing_2} are with rank regularization for $r=10$ (left) and without any regularization, that is $r=m$, (right). The experiment is repeated in 10 trials (resampling the noise), and the mean predictions are plotted in green dashed line, while the $95\%$ confidence region is shaded.

\begin{figure}[ht!]
    \centering
   \includegraphics[scale=0.59]{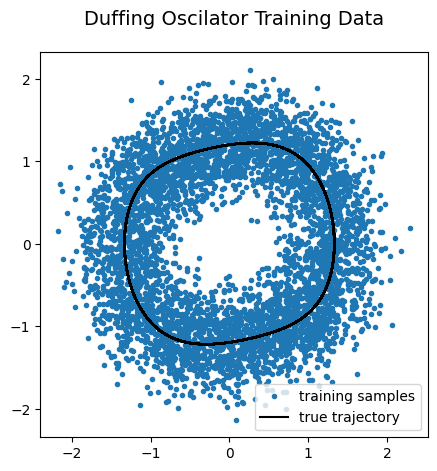}
   \hspace{2.0truecm}
   \includegraphics[scale=0.59]{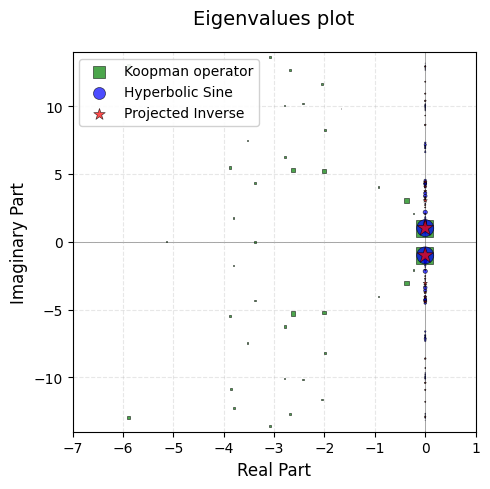}\\
    \caption{Duffing oscilator on a simple attractor: training trajectory and samples (left), and estimated spectrum for three different models (right). True base frequency is $1/2\pi \approx 0.1592$.}
    \label{fig:Duffing_1}
\end{figure}

\begin{figure}[ht!]
    \centering
   \includegraphics[scale=0.53]{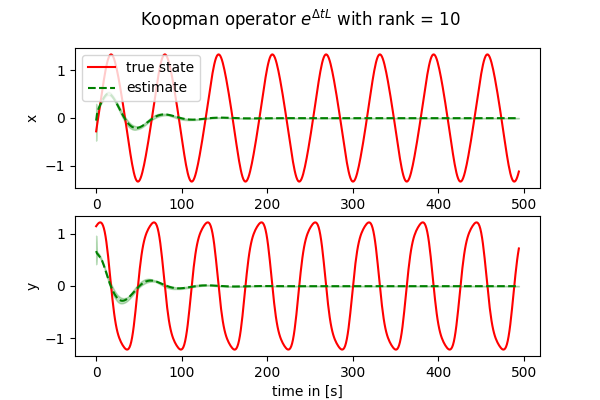}
   \includegraphics[scale=0.53]{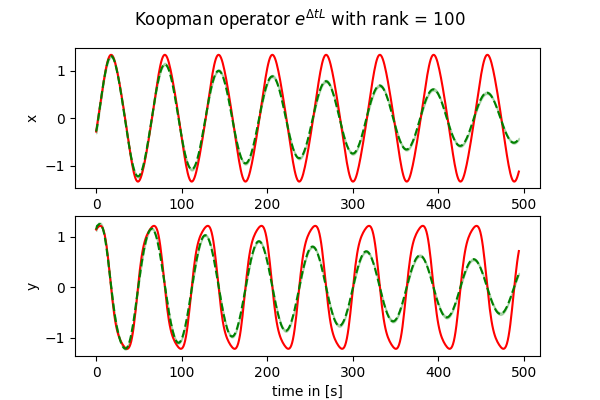}\\
   \vspace{0.3truecm}
   \includegraphics[scale=0.53]{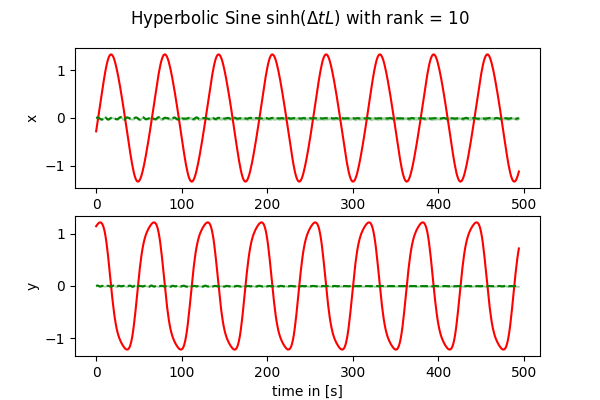}
    \includegraphics[scale=0.53]{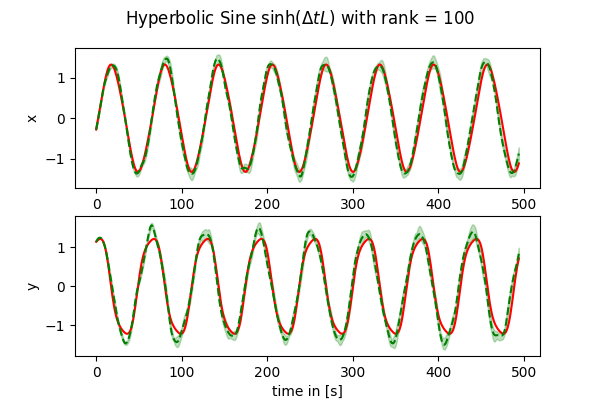}\\
    \vspace{0.3truecm}
   \includegraphics[scale=0.53]{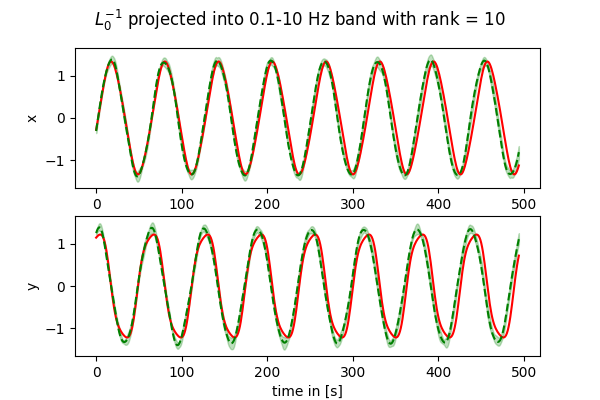}
   \includegraphics[scale=0.53]{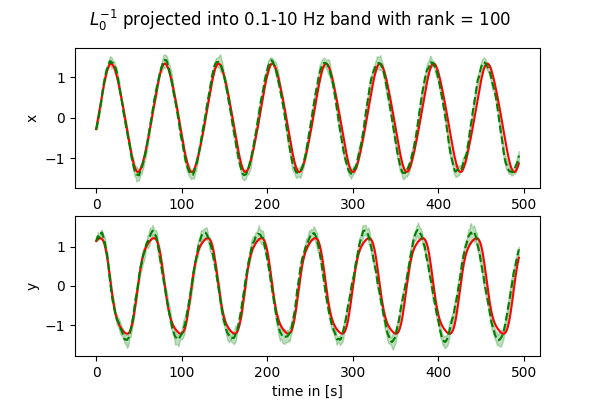}\\
    \caption{Learning the Duffing oscillator on a simple attractor with different Toeplitz based transforms: preserving the spectral properties reduces the model's complexity and improves performance even with highly noisy observations. \textbf{Bottom row:} Toeplitz spectral filtering isolates the relevant frequencies (eigenvalues) in the range $(0.1, 10)$. \textbf{Left:} A low rank ($10$) is sufficient for accurate test-trajectory forecasting. \textbf{Right:} Increasing the rank to $100$ provides no significant improvement and instead seems to slightly increase noise-related variability.
}
    \label{fig:Duffing_2}
\end{figure}

\textbf{Strange attractor (chaotic system).} We now examine the Duffing oscillator with parameters $\alpha=-1$, $\beta=1$, $\gamma=0.5$, $\delta=0.3$ and $\omega=1$, which is known to exhibit a strange attractor and chaotic dynamics \cite{ guckenheimer2013nonlinear}. In this regime, the $\Liishort$-spectrum of the infinitesimal generator $\generator$ becomes continuous, specifically $\Spec(\generator)=i\R$ in the extended phase space, making long-term prediction infeasible. The relevant dynamical information is instead encoded in the behavior of the generator’s resolvent \cite{mezic2020spectrum}. Since no practical resolvent estimation method can achieve uniform error bounds in the operator norm, it is natural to study spectral concentration and decay properties of physically relevant observables. In particular, for a chosen observable $f \in \Lii$, the mapping  
\[
\theta \mapsto \|(\mu+i2\pi\theta-\generator)^{-1} f\|_{\Liishort}, \qquad \theta \in (0,1],
\]
termed the \textit{resolvent response} of $f$, reveals how the system’s dynamics project onto different frequencies.

For the Duffing oscillator, a natural observable is the velocity \(y = \dot{x}\). With the parameters above, the resolvent response of velocity is expected to show pronounced peaks at the forcing frequency and its harmonics, i.e., \(2\pi\theta \approx 1, 2, 3, \dots\), with decaying amplitude, along with possible subharmonic activity near \(1/2\) and a peak around \(\sqrt{2}\) arising from linearized intrawell oscillations \citep{ueda1980steady}.

Conventional data-driven estimators, such as Koopman RRR, are inherently finite-rank and possess a discrete eigenvalue decomposition. These discrete spectra do not directly reflect the continuous spectral structure of \(\generator\) on the attractor. Nevertheless, one can formally transform the eigendecomposition to approximate the resolvent response. In Figure~\ref{fig:Duffing_3} (right), we show the result of applying this approach using a standard transfer operator estimator (Hankel EDMD). The estimator’s spectral decomposition, however, while it recovers the main peak, it over-smooths the velocity’s resolvent response, thus obscuring important propereties (like intrawell oscillations) of the  spectral measure on the chaotic attractor. 

To illustrate how Toeplitz-based spectral filters can yield more accurate estimates of such spectral measures, we trained estimators of \(\anfun(\generator) = (e^{\mu+i2\pi\theta}-e^{\dt L})^{-1}\) and $\anfun(\generator) = (\mu+i2\pi\theta- L)^{-1}$ with $\mu = 0.01$ and over a fine discretization of \(\theta \in (-1,1)\). For each \(\theta\), we approximated the resolvent response of the velocity shown in Figure \ref{fig:Duffing_3}. Notably, both spectral filters recover expected peaks with significantly improved resolution, while the resolvent filter based on the Laplace transform particularly well reveals harmonics of the intrawell oscillations and higher harmonics.

\begin{figure}[h!]
    \centering
   \includegraphics[scale=0.55]{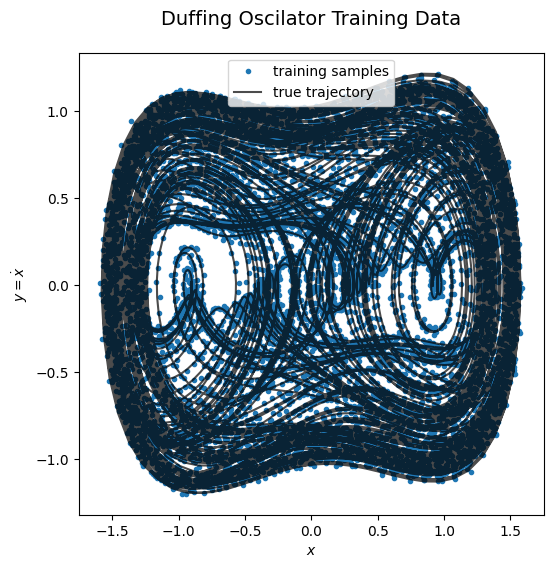}
      \includegraphics[scale=0.55]{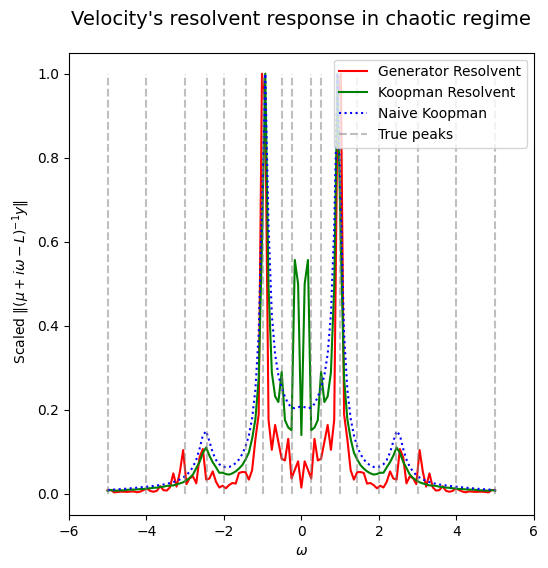}
    \caption{Duffing Oscillator on strange attractor. Training trajectory and samples on the left, and resolvent response of velocity along the Nyquist region on the right.}
    \label{fig:Duffing_3}
\end{figure}

\section{Discussion}
% We presented a first-of-its-kind method to learn continuous Markov semigroups, offering both theoretical guarantees at any time-lag and linear computational complexity in the state dimension, enabling efficient exploration of high-dimensional complex systems. %We presented a new method that is the first to offer both theoretical guarantees at any time-lag and linear computational complexity in the state dimension, enabling the efficient exploration of high-dimensional complex systems. 
% Notably, our method applies to a broad range of Markov processes previously unaddressed, and overcomes the problem of TO-based methods failing to capture slow dynamics when trained on data with high sampling frequency. The main limitation of the current results lies in the assumption of uniform sampling of the (full) state of the system. While the theory can be seamlessly adapted to multiple observations sharing the same non-uniform sampling, an important challenge is to extend it to non-uniformly sampled single trajectory data, as well as to the case of partially observed systems.

We proposed a Toeplitz-based framework for data-driven spectral estimation of Markov
evolution operators from stationary trajectories. By representing analytic transforms of
the generator as $\anfun(L)=\toepsymb(A_{\Delta t})$, the approach unifies the estimation of
transfer operators, resolvents, and frequency-selective filters within a single linear
algebraic structure. Structural properties of the dynamics, such as self-adjointness or
skew-adjointness of the generator, are naturally enforced through symmetry of the Toeplitz
symbol and preserved by construction.
\\
The resulting estimators reduce to weighted time-lagged cross-covariances and admit
efficient implementations via Toeplitz matrix operations, yielding a data-driven analogue
of classical Krylov subspace methods in which polynomial filtering of the evolution
operator is replaced by Toeplitz-weighted time-lag statistics. Statistical consistency follows under
$\beta$-mixing and appropriate truncation of the symbol, clarifying the role of temporal
dependence and sampling. Overall, the framework provides a flexible and scalable approach
to spectral learning for both stochastic and deterministic dynamical systems, including
regimes where direct estimation of the transfer operator is ineffective.
\\
Several directions remain open. Selecting the Toeplitz symbol $\toepsymb$ optimally for a
given task (e.g., forecasting or metastability analysis) is an important problem.
Extending the analysis to noisy observations, partial observability, and non-stationary
data would further broaden applicability. For chaotic deterministic systems with
continuous $\Lii$-spectrum, the extracted spectral features should be interpreted as
resolvent-based or frequency-localized quantities rather than classical eigenvalues; a
precise characterization of this regime remains an interesting direction for future work.

{\bf Acknowledgments.} This work was presented in part at the Applied Linear Algebra Conference in Honor of our dear colleague and friend Zhong-Zhi Bai, whose work has inspired many advances in numerical methods for eigenvalue problems. The authors are grateful for the opportunity to contribute to this important and growing field of research by combining their expertise in numerical linear algebra, statistics, and machine learning.

The work of V.R.K. and M.P. was partially supported by the EU Project ELIAS (grant No. 101120237), and by the European Union – NextGenerationEU and the Italian National Recovery and Resilience Plan through the Ministry of University and Research (MUR), under Project PE0000013 CUP J53C22003010006. The work of K.L. was partially supported by the EU Project ELIAS (grant No. 101120237). 
%We thank the anonymous reviewers for their insightful and valuable feedback.

%\subsubsection*{Acknowledgements}
%All acknowledgments go at the end of the paper, including thanks to reviewers who gave useful comments, to colleagues who contributed to the ideas, and to funding agencies and corporate sponsors that provided financial support. 
%To preserve the anonymity, please include acknowledgments \emph{only} in the camera-ready papers.

%\subsubsection*{References}

%\begin{thebibliography}{}
%\setlength{\itemindent}{-\leftmargin}
%\makeatletter\renewcommand{\@biblabel}[1]{}\makeatother
%\bibliographystyle{abbrvnat}

\bibliographystyle{abbrvnat}  
\bibliography{learning.bib}
%\end{thebibliography}

\end{document}